\documentclass{siamltex}

\usepackage[utf8]{inputenc}
\usepackage{amsmath,amssymb,array,stmaryrd,booktabs}
\usepackage{algorithmic,algorithm}
\usepackage{graphics,graphicx}
\usepackage[single]{accents}
\usepackage{subfigure}

\usepackage{color}
\usepackage[final, framed]{mcode}
\usepackage{caption}

\newcommand{\ls}{\ensuremath{\text{\textsc{LS}}}}
\newcommand{\lsapprox}{\ensuremath{\texttt{\textsc{LSapprox}}}}
\newcommand{\lsrel}[1]{\ensuremath{\text{\textsc{LSapprox}}^{\text{rel}}_{#1}}}
\newcommand{\lsabs}[1]{\ensuremath{\text{\textsc{LSapprox}}^{\text{abs}}_{#1}}}

\title{Optimization of Convex Functions with Random Pursuit\footnotemark[1]}

\author{S.~U. Stich\footnotemark[2] 
\and C.~L. M\"{u}ller\footnotemark[3]
\and B. G\"{a}rtner\footnotemark[4]}

\DeclareMathOperator{\expect}{\mathbb{E}}

\DeclareMathOperator{\sign}{sign}

\newcommand{\lin}[1]{\ensuremath{\left\langle #1 \right\rangle}}

\newcommand{\reals}{\mathbb{R}}
\newcommand{\naturals}{\mathbb{N}}

\newcommand{\normal}{\mathcal{N}}

\providecommand{\norm}[1]{\left\lVert#1\right\rVert}

\usepackage{color}
\definecolor{darkred}{rgb}{.5,0,0}
\definecolor{darkgreen}{rgb}{0,.4,.2}
\definecolor{darkblue}{rgb}{.1,.2,.6}

\begin{document}
\maketitle

\renewcommand{\thefootnote}{\fnsymbol{footnote}}
\footnotetext[1]{The project CG Learning acknowledges the financial support of the Future and Emerging Technologies (FET) programme within the Seventh Framework Programme for Research of the European Commission, under FET-Open grant number: 255827}
\footnotetext[2]{Institute of Theoretical Computer Science, ETH Z\"urich, and Swiss Institute of Bioinformatics, \texttt{sstich@inf.ethz.ch}}
\footnotetext[3]{Institute of Theoretical Computer Science, ETH Z\"urich, and Swiss Institute of Bioinformatics, \texttt{christian.mueller@inf.ethz.ch}}
\footnotetext[4]{Institute of Theoretical Computer Science, ETH Z\"urich, \texttt{gaertner@inf.ethz.ch}}
\renewcommand{\thefootnote}{\arabic{footnote}}

\begin{abstract}
We consider unconstrained randomized optimization of convex objective functions. We analyze the Random Pursuit algorithm, which iteratively computes an approximate solution to the optimization problem by repeated optimization over a randomly chosen one-dimensional subspace. This randomized method only uses zeroth-order information about the objective function and does not need any problem-specific parametrization. We prove convergence and give convergence rates for smooth objectives assuming that the one-dimensional optimization can be solved exactly or approximately by an oracle. A convenient property of Random Pursuit is its invariance under strictly monotone transformations of the objective function. It thus enjoys identical convergence behavior on a wider function class. To support the theoretical results we present extensive numerical performance results of Random Pursuit, two gradient-free algorithms recently proposed by Nesterov, and a classical adaptive step-size random search scheme. We also present an accelerated heuristic version of the Random Pursuit algorithm which significantly improves standard Random Pursuit on all numerical benchmark problems. A general comparison of the experimental results reveals that (i) standard Random Pursuit is effective on strongly convex functions with moderate condition number, and (ii) the accelerated scheme is comparable to Nesterov's fast gradient method and outperforms adaptive step-size strategies.    
\end{abstract}

\begin{keywords} 
continuous optimization, convex optimization, randomized algorithm, line search
\end{keywords}

\begin{AMS}
90C25, 90C56, 68W20, 62L10
\end{AMS}

\pagestyle{myheadings}
\thispagestyle{plain}
\markboth{S.~U. STICH, C.~L. M\"{U}LLER AND B. G\"{A}RTNER}{RANDOM PURSUIT}

\section{Introduction}
\label{sec:intro}
Randomized zeroth-order optimization schemes were among the first algorithms proposed to numerically solve unconstrained optimization problems~\cite{anderson:53,Brooks:1958,rastrigin:63}. These methods are usually easy to implement, do not require gradient or Hessian information about the objective function, and comprise a randomized mechanism to iteratively generate new candidate solutions. In many areas of modern science and engineering such methods are indispensable in the simulation (or black-box) optimization context, where higher-order information about the simulation output is not available or does not exist. Compared to deterministic zeroth-order algorithms such as \emph{direct search} methods \cite{Kolda:2003} or interpolation methods \cite{conn:09} randomized schemes often show faster and more robust performance on ill-conditioned benchmark problems \cite{Auger:2009} and certain real-world applications such as quantum control \cite{Brif:2010} and parameter estimation in systems biology networks \cite{Sun:2012}. While probabilistic convergence guarantees even for non-convex objectives are readily available for many randomized algorithms \cite{Zhigljavsky:2008}, provable \emph{convergence rates} are often not known or unrealistically slow. Notable exceptions can be found in the literature on adaptive step size random search (also known as Evolution Strategies) \cite{Beyer:2001, jaeg:05}, on Markov chain methods for volume estimation, rounding, and optimization \cite{Vempala:2010}, and in Nesterov's recent work on complexity bounds for gradient-free convex optimization~\cite{nesterov:randomgradientfree}. 

Although Nesterov's algorithms are termed ``gradient-free" their working mechanism does, in fact, rely on approximate \emph{directional derivatives} that have to be available via a suitable oracle. We here relax this requirement and investigate a true randomized gradient- and derivative-free optimization algorithm: \emph{Random Pursuit} ($\mathcal{RP}_{\mu}$).  The method comprises two very elementary primitives: a random direction generator and an (approximate) line search routine. We establish theoretical performance bounds of this algorithm for the unconstrained convex minimization problem
\begin{align}
 \label{eq:problem}
 \min f(x) \quad \text{subject to} \quad x \in \reals^n \, , 
\end{align}
where $f$ is a smooth convex function. We assume that there is a global minimum and that the curvature of the function $f$ can bounded by a constant. Each iteration of Random Pursuit consists of two steps: A random direction is sampled uniformly at random from the unit sphere. The next iterate is chosen such as to (approximately) minimize the objective function along this direction. This method ranges among the simplest possible optimization schemes as it solely relies on two easy-to-implement primitives: a random direction generator and an (approximate) one-dimensional line search. A convenient feature of the algorithm is that it inherits the invariance under strictly monotone transformations of the objective function from the line search oracle. The algorithm thus enjoys convergence guarantees even for non-convex objective functions that can be transformed into convex objectives via a suitable strictly monotone transformation.  

Although Random Pursuit is fully gradient- and derivative-free, it can still be understood from the perspective of the classical gradient method. The gradient method ($\mathcal{GM}$) is an iterative algorithm where the current approximate solution $x_k \in \reals^n$ is improved along the direction of the negative gradient with some step size $\lambda_k$: 
\begin{align}
\label{eq:GM}
x_{k+1} = x_k + \lambda_k (-\nabla f(x_k)) \, .
\end{align}
When the descent direction is replaced by a random vector the generic scheme reads
\begin{align}
\label{eq:random}
x_{k+1} = x_k + \lambda_k u \, ,
\end{align}
where $u$ is a random vector distributed uniformly over the unit
sphere. A crucial aspect of the
performance of this randomized scheme is the determination of the step
size. Rastrigin~\cite{rastrigin:63} studied the convergence of this
scheme on quadratic functions for fixed step sizes $\lambda_k$ where
only improving steps are accepted. Many authors observed that variable
step size methods yield faster convergence
\cite{maybach:66,karnopp:63}. Schumer and Steiglitz~\cite{schumer:68}
were among the first to develop an effective step size adaptation rule
which is based on the maximization of the expected one-step progress
on the sphere function. A similar analysis has been independently
obtained by Rechenberg for the (1+1)-Evolution Strategy
($\mathcal{ES}$) \cite{Rechenberg:1973}. Mutseniyeks and Rastrigin
proposed to choose the step size such as to minimize the function
value along the random direction \cite{rastrigin:64}. This algorithm
is identical to Random Pursuit with an \emph{exact} line search.
Convergence analyses on strongly convex functions have been provided
by Krutikov~\cite{krutikov:83} and Rappl~\cite{rappl:89}. Rappl proved
linear convergence of $\mathcal{RP}_{\mu}$ without giving exact convergence
rates. Krutikov showed linear convergence in the special case where
the search directions are given by $n$ linearly independent vectors
which are used in cyclic order.

Karmanov~\cite{Karmanov:1974b,karmanov:74,Ziel:83} already conducted
an analysis of Random Pursuit on general convex functions. Thus far,
Karmanov's work has not been recognized by the optimization community
but his results are very close to the work presented here. We enhance
Karmanov's results in a number of ways: (i) we prove expected
convergence rates also under \emph{approximate} line search; (ii) we
show that continuous sampling from the unit sphere can be replaced
with discrete sampling from the set $\{\pm e_i: i=1,\ldots,n\}$ of
signed unit vectors, without changing the expected convergence rates;
(iii) we provide a large number of experimental results, showing that
Random Pursuit is a competitive algorithm in practice; (iv) we
introduce a heuristic improvement of Random Pursuit that is even
faster on all our benchmark functions; (v) we point out that Random
Pursuit can also be applied to a number of relevant non-convex
functions, without sacrificing any theoretical and practical
performance guarantees. On the other hand, while we prove fast
convergence only in expectation, Karmanov's more intricate analysis also
yields fast convergence with high probability.
 
Polyak ~\cite{polyak:87} describes step size rules for the
closely related randomized gradient descent scheme:
\begin{align}
\label{eq:random2}
x_{k+1} = x_k + \lambda_k \frac{f(x_k + \mu_k u) - f(x_k)}{\mu_k} u \, ,
\end{align}
where convergence is proved for $\mu_k \to 0$ but no convergence rates are established. Nesterov~\cite{nesterov:randomgradientfree} studied different variants of method~\eqref{eq:random2} and its accelerated versions for smooth and non-smooth optimization problems. He showed that scheme~\eqref{eq:random2} is at most $O(n^2)$ times slower than the standard (sub-)gradient method. The use of exact directional derivatives reduces the gap further to $O(n)$. For smooth problems the method is only $O(n)$ slower than the standard gradient method and accelerated versions are $O(n^2)$ slower than fast gradient methods. 

Kleiner et al.~\cite{kleiner:rcp} studied a variant of algorithm~\eqref{eq:random} for unconstrained semidefinite programming: Random Conic Pursuit. There, each iteration comprises two steps: (i) the algorithm samples a rank-one matrix (not necessarily uniformly) at random; (ii) a two-dimensional optimization problem is solved that consists of finding the optimal linear combination of the rank-one matrix and the current semidefinite matrix. The solution determines the next iterate of the algorithm. In the case of trace-constrained semidefinite problems only a one-dimensional line search is necessary. Kleiner and co-workers proved convergence of this algorithm when directions are chosen uniformly at random. The dependency between convergence rate and dimension are, however, not known. Nonetheless, their work greatly inspired our own efforts which is also reflected in the name ``Random Pursuit" for the algorithm under study.

The present article is structured as follows. In Section \ref{sec:RP} we present the Random Pursuit algorithm with approximate line search. We introduce the necessary notation and formulate the assumptions on the objective function. In Section \ref{sec:expectation} we derive a number of useful results on the expectation of scaled random vectors. In Section \ref{sec:generic} we calculate the expected one-step progress of Random Pursuit with approximate line search ($\mathcal{RP}_\mu$). We show that (besides some additive error term) this progress is by a factor of $O(n)$ worse than the one-step progress of the gradient method. These results allow us to derive the final convergence results in Section \ref{sec:convergence}. We show that $\mathcal{RP}_{\mu}$ meets the convergence rates of the standard gradient method up to a factor of $O(n)$, i.e., linear convergence on strongly convex functions and convergence rate $1/k$ for general convex functions. The linear convergence on strongly convex functions is best possible: For the sphere function our method meets the lower bound~\cite{jaeg:07}. 
For strongly convex objective functions the method is robust against small absolute or relative errors in the line search. In Section \ref{sec:experiments} we present numerical experiments on selected test problems. We compare $\mathcal{RP}_\mu$ with a fixed step size gradient method and a gradient scheme with line search, Nesterov's random gradient scheme and its accelerated version~\cite{nesterov:randomgradientfree}, an adaptive step size random search, and an accelerated heuristic version of $\mathcal{RP}_\mu$. In Section \ref{sec:concl} we discuss the theoretical and numerical results as well as the present limitations of the scheme that may be alleviated by more elaborate randomization primitives. We also provide a number of promising future research directions.


\section{The Random Pursuit (RP) Algorithm}
\label{sec:RP}
We consider problem~\eqref{eq:problem} where $f$ is a differentiable convex function with bounded curvature (to be defined below). 
The algorithm $\mathcal{RP}_{\mu}$ is a variant of scheme~\eqref{eq:random} where the step sizes are determined by a line search. 
Formally, we define the following oracles:
\begin{definition}[Line search oracle]
\label{def:ls}
For $x\in \reals^n$, a convex function $f$, and a direction $u \in S^{n-1} = \{y \in \reals^n \colon \norm{y}_2=1\}$, a function $\ls \colon \reals^n \times S^{n-1} \to \reals$ with
\begin{align}
\ls(x,u) \in \displaystyle \arg \min_{h \in \reals} f(x + hu) \label{def:exactlinesearch} 
\end{align}
is called an \emph{exact line search oracle}. (Here, the $\arg\min$ is
not assumed to be unique, so we consider it as a set from which
$\ls(x,u)$ selects a 
well-defined element.)
For accuracy $\mu \geq 0$ the functions $\lsrel{\mu}$ and $\lsabs{\mu}$ with
\begin{align}
\ls(x,u) - \mu \leq \lsabs{\mu}(x,u) \leq \ls(x,u) + \mu, \quad \text{and}  \label{def:abslinesreach}, \\
s \cdot \max\{0,(1-\mu)\} \cdot \ls(x,u) \leq s \cdot \lsrel{\mu}(x,u) \leq s \cdot \ls(x,u), \label{def:rellinesearch}
\end{align}
where $s= \sign(\ls(x,u))$, are, respectively, \emph{absolute} and \emph{relative}, \emph{approximate line search oracles}. 
By $\lsapprox_\mu$, we denote any of the two. 
\end{definition}

This means that we allow an inexact line search to return a value $\tilde h$
close to an optimal value $h^*=\ls(x,u)$. To simplify subsequent calculations,
we also require that $\tilde h\leq h^*$ in the case of relative approximation (cf.~\eqref{def:rellinesearch}),
but this requirement is not essential. As the optimization problem~\eqref{def:exactlinesearch} cannot be solved exactly in most cases, we will describe and analyze our algorithm by means of the two latter approximation routines. 

The formal definition of algorithm $\mathcal{RP}_{\mu}$ is shown in Algorithm~\ref{alg:rp}. 
At iteration $k$ of the algorithm a direction $u \in S^{n-1}$ is chosen uniformly at random and the next iterate $x_{k+1}$ is calculated from the current iterate $x_k$ as 
\begin{align}
x_{k+1} := x_k + \lsapprox_\mu(x_k,u) \cdot u. \label{eq:3}
\end{align}

\renewcommand{\algorithmicrequire}{\textbf{Input:}}
\renewcommand{\algorithmicensure}{\textbf{Output:}}
\algsetup{indent=2em}
\begin{algorithm}[htb]
\small
\caption{Random Pursuit ($\mathcal{RP}_\mu$)}
\label{alg:rp}
%
\begin{algorithmic}[1]
\REQUIRE $\begin{array}{ll} %
                 \text{A problem of the form~\eqref{eq:problem}} &
                 N \in \naturals : \text{number of iterations} \\
                 x_0: \text{an initial iterate}  & \mu > 0: \text{line search accuracy}%
          \end{array}$
\ENSURE Approximate solution $x_N$ to \eqref{eq:problem}.
\medskip
\FOR {$k \leftarrow 0$ to $N-1$}
\STATE choose $u_k$ uniformly at random from $S^{n-1}$
\STATE Set $x_{k+1} \leftarrow x_k + \lsapprox_\mu(x_k,u_k) \cdot u_k$
\ENDFOR
\RETURN $x_N$
\end{algorithmic}
\end{algorithm}
This algorithm only requires function evaluations if the line search $\lsapprox_\mu$ is implemented appropriately (see \cite{Boef:2007} and references therein). 
No additional first or second-order information of the objective is needed. 
Note also that besides the starting point no further input parameters describing function properties (e.g. curvature constant, etc.) are necessary. 
The actual run time will, however, depend on the specific properties of the objective function.

\subsection{Discrete Sampling} 
\label{subsec:discSampl}
As our analysis below reveals, the random vector $u_k$ enters the analysis only in terms of expectations of the form $\expect[\lin{x,u_k}u_k]$ and $\expect[\norm{\lin{x,u_k}u_k}^2]$. In Lemmas~\ref{lemma:asphere}
and~\ref{lemma:aunit} we show that these expectations are the same for $u_k\sim S^{n-1}$ and $u_k\sim \{\pm e_i:i=1,\ldots,n\}$, the set of signed unit vectors (here and in the following, the notation $x \sim A$ for a set $A$, denotes that $x$ is distributed according to the uniform distribution on $A$). It follows that continuous sampling from
$S^{n-1}$ can be replaced with discrete sampling from $\{\pm
e_i:i=1,\ldots,n\}$ without affecting our guarantees on the expected
runtime. Under this modification, fast convergence still holds with
high probability, but the bounds get worse~\cite{karmanov:74}.

\subsection{Quasiconvex Functions} 
If $f$ and $g$ are functions, $g$ is called a \emph{strictly monotone transformation} of $f$ if
\[f(x)<f(y)~~\Leftrightarrow~~ g(f(x))<g(f(y)), \quad x,y\in\reals^n.\] This implies that the distribution of $x_k$ in $\mathcal{RP}_{\mu}$ is the same for the function $f$ and the function
$g\circ f$, if $g$ is a strictly monotone transformation of $f$. This follows from the fact that the result of the line search given in Definition \ref{def:ls} is invariant under strictly monotone transformations.

This observation allows us to run $\mathcal{RP}_{\mu}$ on any strictly
monotone transformation of any convex function $f$, with the same
theoretical performance as on $f$ itself. 
The functions obtainable in this way form a subclass of the class of \emph{quasiconvex functions}, and they include some non-convex functions as well. In
Section~\ref{sec:funnel} we will experimentally verify the invariance of $\mathcal{RP}_{\mu}$ under strictly monotone transformations on one instance of a quasiconvex function.

\subsection{Function Basics}
We now introduce some important inequalities that are useful for the subsequent analysis. 
We always assume that the objective function is differentiable and convex. 
The latter property is equivalent to
\begin{align}
\label{def:convex1}
f(y) \geq f(x) + \lin{\nabla f(x), y-x}, \quad x, y \in \reals^n \, .
\end{align} 
We also require that the curvature of $f$ is bounded. 
By this we mean that for some constant $L_1$, 
\begin{align}
\label{eq:lipschitz}
f(y)-f(x) - \lin{\nabla f(x), y-x} \leq \frac{1}{2} L_1 \norm{x-y}^2, \quad x,y \in \reals^n \, .
\end{align}
We will also refer to this inequality as the \emph{quadratic upper
  bound}. It means that the deviation of $f$ from any of its linear
approximations can be bounded by a quadratic function.
We denote by $C^1_{L_1}$ the class of differentiable and convex functions for with the quadratic upper bound holds with the constant $L_1$.

A differentiable function is \emph{strongly convex} with parameter $m$ if the \emph{quadratic lower bound}
\begin{align}
\label{def:stronconvex}
f(y) - f(x) - \lin{\nabla f(x), y-x} \geq \frac{m}{2} \norm{y-x}^2 \, , \quad x,y \in \reals^n \, ,
\end{align}
holds. 
Let $x^*$ be the unique minimizer of a strongly convex function $f$ with parameter $m$. 
Then equation~\eqref{def:stronconvex} implies this useful relation:
\begin{align}
\label{eq:quadraticlower}
\frac{m}{2} \norm{x-x^*}^2 \leq f(x) - f(x^*) \leq \frac{1}{2m}\|\nabla f(x)\|^2 \, , \quad \forall x \in \reals^n \, .
\end{align}
The former inequality uses $\nabla f(x^*)=0$, and the latter one follows from (\ref{def:stronconvex}) via
\begin{eqnarray*}
f(x^*) &\geq& f(x) + \lin{\nabla f(x), x^*-x} + \frac{m}{2} \norm{x^*-x}^2 \\
&\geq& f(x)+\min_{y\in\reals^n}\left(\lin{\nabla f(x), y-x} + \frac{m}{2} \norm{y-x}^2\right) = f(x) -\frac{1}{2m}\|\nabla f(x)\|^2
\end{eqnarray*}
by standard calculus.
\section{Expectation of Scaled Random Vectors}
\label{sec:expectation}
We now study the projection of a fixed vector $x$ onto a random vector $u$. 
This will help analyze the expected progress of Algorithm~\ref{alg:rp}. 
We start with the case $u \sim \normal(0,I_n)$ and 
then extend it to $u\sim S^{n-1}$.
Throughout this section, let $x \in \reals^n$ be a fixed vector and $u \in \reals^n$ a random vector drawn according to some distribution.
We will need the following facts about the moments of the standard normal distribution.
\begin{lemma}\label{lemma:normalbasic}
\begin{itemize}
\item[(i)] Let $\nu \sim \normal(0,1)$ be drawn from the standard normal distribution over the reals. Then 
\begin{align*}
\expect[\nu] &= \expect[\nu^3] = 0 \, , & \expect[\nu^2]&=1 \, , & \expect[\nu^4]&=3 \, .
\end{align*}
\item[(ii)] Let $u\sim\normal(0,I_n)$ be drawn from the standard normal distribution over $\reals^n$. Then
\begin{align*}
\expect_u[uu^T] &= I_n \, , & \expect_u[(u^Tu)uu^T]&= (n+2)I_n \, .
\end{align*}
\end{itemize}
\end{lemma}
{\em Proof}.
Part (i) is standard, and the latter two matrix equations easily
follow from (i) via
\begin{align*}
(uu^T)_{ij} &= u_iu_j\, , & \left((u^Tu)uu^T\right)_{ij} &= u_iu_j\sum_{k}u_k^2 \, . \qquad \endproof
\end{align*}

\begin{lemma}[Normal distribution]
\label{lemma:anormal}
Let $u \sim \normal(0, I_n)$. Then
\begin{align*}
\expect_u \left[ \lin{x,u} u \right] = x \, , \quad &\text{and} \quad \expect_u \left[ \norm{\lin{x,u} u}^2 \right] = (n+2) \norm{x}^2 \, .
\end{align*}
\end{lemma}
\begin{proof}
We calculate
\begin{align*}
\expect_u \left[ \lin{x,u} u \right] = \expect_u[uu^Tx]
= \expect_u[uu^T]x = x,
\end{align*}
by Lemma~\ref{lemma:normalbasic}(ii). For the second moment we get
\begin{align*}
\expect_u \left[\norm{\lin{x,u} u}^2\right] = \expect_u[x^T(u^Tu)uu^Tx]
= x^T\expect_u[(u^Tu)uu^T]x = (n+2)\norm{x}^2 \, ,
\end{align*}
again using Lemma~\ref{lemma:normalbasic}(ii). 
\end{proof}
\medskip

\begin{lemma}[Spherical distribution]
\label{lemma:asphere}
Let $u\sim S^{n-1}$. Then
\begin{align*}
\expect_u \left[ \lin{x,u} u \right] = \frac{1}{n} x \, , \quad &\text{and} \quad \expect_u \left[ \norm{\lin{x,u}u}^2 \right] =\expect_u \left[\lin{x,u}^2 \right]= \frac{1}{n} \norm{x}^2 \, .
\end{align*}
\end{lemma}
{\em Proof}.
Let $v \sim \normal(0, I_n)$. We observe that the random vector $w=v/\norm{v}$ has the same distribution as $u$. In particular,
\begin{equation}\label{eq:uv}
\expect_u \left[uu^T\right] = \expect_v \left[ \frac{vv^T}{\norm{v}^2} \right]
= \frac{ \expect_v \left[vv^T \right]}{ \expect_v \left[ \norm{v}^2 \right]}
= \frac{I_n}{n} \, ,
\end{equation}
where we have used that the two random variables $\frac{vv^T}{\norm{v}^2}$ and $\norm{v}^2$ are independent (see~\cite{heijmans:ratio}), along with 
\[
\expect_v\left[vv^T\right]=I_n, \quad \expect_v \left[ \norm{v}^2 \right] = n \, ,
\]
a consequence of Lemma~\ref{lemma:normalbasic}.
Now we use (\ref{eq:uv}) to compute
\begin{align*}
\expect_u \left[ \lin{x,u}  u \right] = \expect_u \left[uu^T\right]x
= \frac{I_n}{n}x = \frac{1}{n}x 
\end{align*}
and
\begin{align*}
\expect_u \left[\lin{x,u}^2 \right]
&= \expect_u \left[x^Tuu^Tx\right] = x^T\expect_u \left[uu^T\right]x =
x^T\frac{I_n}{n}x = \frac{1}{n} \norm{x}^2 \, . \qquad\endproof
\end{align*}

The same result can be derived when the vector $u$ is chosen to be a
random signed unit vector.
\begin{lemma}
\label{lemma:aunit}
Let $u\sim U := \{\pm e_i: i=1,\ldots, n\}$ where $e_i$ denotes
  the $i$-th standard unit vector in $\reals^n$. Then
\begin{align*}
\expect_u \left[ \lin{x,u} u \right] = \frac{1}{n} x \, , \quad &\text{and} \quad \expect_u \left[ \norm{\lin{x,u} u}^2 \right] = \expect_u \left[\lin{x,u}^2 \right] = \frac{1}{n} \norm{x}^2.
\end{align*}
\end{lemma}
{\em Proof}.
We calculate
\begin{align*}
\expect_u \left[ \lin{x,u}  u \right] = \frac{1}{2n} \sum_{u \in U} \lin{x, u}  u = \frac{1}{n} \sum_{i=1}^n x_i e_i = \frac{1}{n} x \, ,
\end{align*}
and similarly
\begin{align*}
\expect_u \left[ \lin{x,u}^2 \right] = \frac{1}{2n} \sum_{u \in U}\lin{x, u}^2 = \frac{1}{n} \sum_{i=1}^n x_i^2 = \frac{1}{n} \norm{x}^2 \, . \qquad \endproof
\end{align*}

\section{Single Step Progress}
\label{sec:generic}
To prepare the convergence proof of Algorithm~$\mathcal{RP}_\mu$
in the next section, we study the expected progress in a single step,
which is the quantity
\[
\expect \left[ f(x_{k+1}) \mid x_k \right] \, .
\]
It turns out that we need to proceed differently, depending on whether
the function under consideration is strongly convex (the easier case)
or not. We start with a preparatory lemma for both cases.  We first
analyze the case when an approximate line search with absolute error
is applied. Using an approximate line search with relative error will
be reduced to the case of an exact line search.

\subsection{Line Search with Absolute Error}
\begin{tiny}
~
\end{tiny}
\begin{lemma}[Absolute Error]
\label{lemma:onestepabs}
Let $f \in C^1_{L_1}$ and let $x_k \in \reals^n$ be the current iterate and $x_{k+1} \in \reals^n$
the next iterate generated by algorithm $\mathcal{RP}_\mu$ with
absolute line search accuracy $\mu$. For every positive $h\in\reals$ and
every point $z \in \reals^n$ we have
\begin{align*}
\expect \left[ f(x_{k+1}) \mid x_k \right] %
& \leq f(x_k) + \frac{h}{n} \lin{\nabla f(x_k),z-x_k}  + \frac{L_1 h^2}{2n} \norm{z-x_k}^2 + \frac{L_1 \mu^2}{2} \, .
\end{align*}
\end{lemma}
\begin{proof}
  Let $x'_{k+1}:=x_k+\ls(x_k,u_k)u_k$ be the exact line search
  optimum. Here, $u_k\in S^{n-1}$ is the chosen search direction. By
  definition of the approximate line search~\eqref{def:abslinesreach},
  we have
\begin{align}
f(x_{k+1}) &\leq \displaystyle \max_{|\nu | \leq \mu} f(x'_{k+1} + \nu  u_k) \notag \\ 
&\stackrel{(\ref{eq:lipschitz})}{\leq} f(x'_{k+1}) + \max_{|\nu | \leq \mu} \left[ \underbrace{\lin{\nabla f(x'_{k+1}), \nu u_k}}_{0} + \frac{L_1}{2} \nu^2 \right] \notag \\
&= f(x'_{k+1}) +  \frac{L_1 \mu^2}{2}\, ,
\label{eq:cont1}
\end{align}
where we used the quadratic upper bound~\eqref{eq:lipschitz} in the second inequality with $x=x'_{k+1}$ and $y=x'_{k+1}+\nu  u_k$. 

Since $x'_{k+1}$ is the exact line search optimum, we in particular have
\begin{align}
\label{eq:ls_vs_alpha}
f(x'_{k+1})\leq f(x_k+t_ku_k) \leq f(x_k) + \lin{\nabla f(x_k),t_ku_k}+
\frac{L_1t_k^2}{2} \, , \quad\forall t_k\in\reals,
\end{align}
where we have applied~\eqref{eq:lipschitz} a second time. Putting 
together~\eqref{eq:cont1} and~\eqref{eq:ls_vs_alpha}, and taking 
expectations, we get
\begin{align}
\label{eq:almostdone}
\expect_{u_k} \left[ f(x_{k+1}) \mid x_k \right] %
& \leq f(x_k) + \expect_{u_k} \left[ \lin{\nabla f(x_k),t_ku_k}+\frac{L_1 t_k^2}{2}\mid x_k \right] + \frac{L_1 \mu^2}{2} \, .
\end{align}
Now it is time to choose $t_k$ such that we can control the expectations
on the right-hand side. We set
\[t_k := h\lin{z-x_k,u_k},\]
where $h>0$ and $z\in\reals^n$ are the ``free parameters'' of the lemma.
Via Lemma~\ref{lemma:asphere}, this entails
\[
\expect_{u_k} \left[t_ku_k\right] = \frac{h}{n}(z-x_k) \, , \quad
\expect_{u_k} \left[t_k^2\right] = \frac{h^2}{n}\norm{z-x_k}^2 \, ,
\]
and the lemma follows.
\end{proof}

\subsection{Line Search with Relative Error}
In the case of relative line search error, we can prove a variant of
Lemma~\ref{lemma:onestepabs} with a denominator $n'$ slightly larger
than $n$. As a result, the analysis under relative line search error
reduces to the analysis of exact line search (approximate line search
error $0$) in a slightly higher dimension; in the sequel, we will
therefore only deal with absolute line search error.

\begin{lemma}[Relative Error]
\label{lemma:onesteprel}
Let $f \in C^1_{L_1}$ and let $x_k \in \reals^n$ be the current iterate and $x_{k+1} \in \reals^n$
the next iterate generated by algorithm $\mathcal{RP}_\mu$ with
relative line search accuracy $\mu$. For every positive $h\in\reals$ and
every point $z \in \reals^n$ we have
\begin{align*}
\expect \left[ f(x_{k+1}) \mid x_k \right] %
& \leq f(x_k) + \frac{h}{n'} \lin{\nabla f(x_k),z-x_k}  + \frac{L_1 h^2}{2n'} \norm{z-x_k}^2 \, , 
\end{align*}
where $n'=n/(1-\mu)$.
\end{lemma}

\begin{proof} By the definition (\ref{def:rellinesearch}) of relative
  line search error, $x_{k+1}$ is a convex combination of $x_k$ and
  $x'_{k+1}$, the exact line search optimum. More precisely, we can
  compute that
\[
x_{k+1} = (1-\gamma) x_k + \gamma x'_{k+1} \, , 
\]
where $\gamma\geq 1-\mu$. By convexity of $f$, we thus have 
\[
f(x_{k+1}) \leq (1-\gamma) f(x_k) + \gamma f(x'_{k+1})
\leq \mu f(x_k) + (1-\mu) f(x'_{k+1}) \, ,
\]
since $f(x'_{k+1})\leq f(x_k)$. Hence
\begin{equation}\label{eq:suffdecr}
\expect \left[ f(x_{k+1}) \mid x_k \right] \leq \mu f(x_k) +
(1-\mu) \expect\left[ f(x'_{k+1})\mid x_k\right] \, .
\end{equation}
Using Lemma~\ref{lemma:onestepabs} with absolute line search error $0$ 
yields a bound for the latter term:
\[
\expect \left[ f(x'_{k+1}) \mid x_k \right] \leq f(x_k) + \frac{h}{n}
\lin{\nabla f(x_k),z-x_k} + \frac{L_1 h^2}{2n} \norm{z-x_k}^2 \, .
\] 
Putting this together with (\ref{eq:suffdecr}) yields 
\[
\expect \left[ f(x_{k+1}) \mid x_k \right] \leq f(x_k) + (1-\mu)
\left( \frac{h}{n}
\lin{\nabla f(x_k),z-x_k} + \frac{L_1 h^2}{2n} \norm{z-x_k}^2 \right) \, , 
\]
and with $n'=n/(1-\mu)$, the lemma follows.
\end{proof}

\subsection{Towards the Strongly Convex Case}
Here we use $z=x_k - \nabla f(x_k)$ in Lemma~\ref{lemma:onestepabs}.

\begin{corollary}
\label{cor:onestep_grad}
Let $f \in C^1_{L_1}$ and let $x_k \in \reals^n$ be the current iterate and $x_{k+1} \in \reals^n$ the next iterate generated by algorithm $\mathcal{RP}_\mu$ with absolute line search accuracy $\mu$. For any positive $h_k \leq \frac{1}{L_1}$ it holds that
\begin{align*}
\expect \left[ f(x_{k+1}) \mid x_k \right] %
& \leq f(x_k) - \frac{h_k}{2n} \norm{\nabla f(x_k)}^2 + \frac{L_1 \mu^2}{2}\, .
\end{align*}
\end{corollary}

{\em Proof}.
Lemma~\ref{lemma:onestepabs} with $z= x_k - \nabla f(x_k)$ yields
\begin{align*}
\expect \left[ f(x_{k+1}) \mid x_k \right] %
& \leq f(x_k) - \frac{h_k}{n} \lin{\nabla f(x_k),\nabla f(x_k)}  + \frac{L_1 h_k^2}{2n} \norm{ \nabla f(x_k)}^2 + \frac{L_1 \mu^2}{2} \, .
\end{align*}
We conclude
\begin{align*}
\expect \left[ f(x_{k+1})\mid x_k \right] %
\leq  f(x_k)  -  \underbrace{ \frac{h_k}{n} \left(1- \frac{L_1 h_k}{2} \right)}_{\geq \frac{h_k}{2 n}} \norm{\nabla f(x_k)}^2 + \frac{L_1 \mu^2}{2} \, . \qquad \endproof
\end{align*}

\subsection{Towards the General Convex Case}
For this case, we apply Lemma~\ref{lemma:onestepabs} with $z=x^*$.
\begin{corollary}
\label{cor:onestep_opt}
Let $f \in C^1_{L_1}$ and let $x_k \in \reals^n$ be the current iterate and $x_{k+1} \in \reals^n$ the next iterate generated by algorithm $\mathcal{RP}_\mu$ with absolute line search accuracy $\mu$. Let $x^* \in \reals^n$ be one of the minimizers of the function $f$. For any positive $h_k \geq 0$ it holds that
\begin{align*}
\expect \left[ f(x_{k+1})-f(x^*) \mid x_k \right] %
& \leq (1-\frac{h_k}{n})  \left(f(x_k)-f(x^*) \right)  + \frac{L_1 h_k^2}{2n} \norm{x^*-x_k}^2 + \frac{L_1 \mu^2}{2}\, . 
\end{align*}
\end{corollary}

\begin{proof}
We use Lemma~\ref{lemma:onestepabs} with $z=x^*$ and apply convexity
(\ref{def:convex1}) to bound the term $\lin{\nabla f(x_k),x^*-x_k}$
from above by $f(x^*)-f(x_k)$. Subtracting $f(x^*)$ from both sides
yields the inequality of the corollary. 
\end{proof}

\section{Convergence Results}
\label{sec:convergence}
Here we use the previously derived bounds on the expected single step
progress (Corollaries~\ref{cor:onestep_grad}
and~\ref{cor:onestep_opt}) to show convergence of the algorithm.

\subsection{Convergence Analysis for Strongly Convex Functions}
We first prove that algorithm $\mathcal{RP}_{\mu}$ converges linearly
in expectation on strongly convex functions.  Despite strong convexity
being a global property, it is sufficient if the function is strongly
convex in the neighborhood of its minimizer (see
Theorem~\ref{thm:localstrong}).
\begin{theorem}
\label{thm:stongconvex}
Let $f \in C^1_{L_1}$ and let $f$ be strongly convex with parameter $m$, and consider the sequence $\{x_k\}_{k\geq0}$ generated by $\mathcal{RP}_{\mu}$ with absolute line search accuracy $\mu$. Then for any $N \geq 0$, we have
\begin{align*}
\expect \left[f(x_N) - f(x^*) \right] \leq \left(1- \frac{m}{L_1 n} \right)^N \left(f(x_0) - f(x^*)\right) + \frac{L_1^2 n \mu^2}{2m} \, .
\end{align*}
\end{theorem}

\begin{proof}
We use Corollary~\ref{cor:onestep_grad} with $h=\frac{1}{L_1}$ and the quadratic lower bound 
to estimate the progress in one step as
\begin{align*}
\expect \left[ f(x_{k+1}) - f(x^*) \mid x_k \right] %
&\leq   f(x_k) - f(x^*) -  \frac{1}{2 n L_1} \norm{\nabla f(x_k)}^2 + \frac{L_1 \mu^2}{2} \\
&\stackrel{\eqref{eq:quadraticlower}}{\leq} \left(1- \frac{m}{n L_1} \right)  \left(f(x_k) - f(x^*)\right) + \frac{L_1 \mu^2}{2} \, .
\end{align*}
After taking expectations (over $x_k$), the tower property of conditional
expectations yields the recurrence
\[
\expect \left[ f(x_{k+1}) - f(x^*)\right] \leq \left(1- \frac{m}{n L_1} \right)\expect \left[ f(x_{k}) - f(x^*)\right]+ \frac{L_1 \mu^2}{2} \, .
\]
This implies
\[
\expect \left[ f(x_{N}) - f(x^*)\right] \leq \left(1-\frac{m}{n L_1}\right)^N 
\left(f(x_{0}) - f(x^*)\right)+ \omega(N) \frac{L_1 \mu^2}{2} \, ,
\]
with
\begin{align*}
\omega(N) :=  \sum_{i=0}^{N-1} (1-\frac{m}{L_1 n})^i \leq \frac{L_1 n}{m} \, .
\end{align*} 
The bound of the theorem follows. 
\end{proof}

We remark that by strong convexity the error $\norm{x_N-x^*}$ can also
be bounded using the results of this theorem. Thus, the algorithm does not only converge in terms of function value, but also in terms of the solution itself.

Each strongly convex function has a unique minimizer $x^*$. Using the quadratic lower bound~\eqref{eq:quadraticlower} we recall that:
\begin{align}
\label{eq:localstrong}
f(x)- f(x^*) \geq \frac{m}{2} \norm{x-x^*}^2, \quad \forall x \in \reals^n.
\end{align}
It turns out that instead of strong convexity~\eqref{def:stronconvex} the weaker condition~\eqref{eq:localstrong} is sufficient to have linear convergence. 

\begin{theorem}
\label{thm:localstrong}
Let $f \in C^1_{L_1}$ and suppose $f$ has a unique minimizer $x^*$ satisfying~\eqref{eq:localstrong} with parameter $m$. Consider the sequence $\{x_k\}_{k\geq0}$ generated by $\mathcal{RP}_{\mu}$ with absolute line search accuracy $\mu$. Then for any $N \geq 0$, we have
\begin{align*}
\expect \left[f(x_N) - f(x^*) \right] \leq \left(1- \frac{m}{4 L_1 n} \right)^N \left(f(x_0) - f(x^*)\right) + \frac{2 L_1^2 n \mu^2}{m} \, .
\end{align*}
\end{theorem}
\begin{proof}
To see this we use Corollary~\ref{cor:onestep_opt} with property~\eqref{eq:localstrong} to get
\begin{align*}
\expect \left[ f(x_{k+1})-f(x^*) \mid x_k \right] %
& \leq \left(1-\frac{h_k}{n}\right)  \left(f(x_k)-f(x^*) \right)  + \frac{L_1 h_k^2}{2n} \norm{x_k-x^*}^2 + \frac{L_1 \mu^2}{2} \\
& \leq \left(1-\frac{h_k}{n} + \frac{L_1 h_k^2}{mn}  \right)  \left(f(x_k)-f(x^*) \right) + \frac{L_1 \mu^2}{2} \, .
\end{align*}
Setting $h_k$ to $\frac{m}{2L_1}$, the term in the left bracket becomes 
$(1- \frac{m}{4 L_1 n})$. Now the proof continues as the proof of
Theorem~\ref{thm:stongconvex}.
\end{proof}

\subsection{Convergence Analysis for Convex Functions}
We now prove that algorithm $\mathcal{RP}_{\mu}$ converges in expectation on smooth (not necessarily strongly) convex functions. The rate is, however, not linear anymore.

\begin{theorem}
\label{thm:convex_bounded}
Let $f \in C^1_{L_1}$ and let $x^*$ a minimizer of $f$, and let the sequence $\{x_k\}_{k\geq0}$ be generated by $\mathcal{RP}_{\mu}$ with absolute line search accuracy $\mu$. 
Assume there exists $R$, s.t. $\norm{y-x^*} < R$ for all $y$ with $f(y) \leq f(x_0)$. 
Then for any $N \geq 0$, we have
\begin{align*}
\expect \left[ f(x_N)- f(x^*) \right] \leq \frac{Q}{N+1} + \frac{N L_1 \mu^2}{2} \, ,
\end{align*}
where
\begin{align*}
Q=\max\left\{2 n L_1 R^2, f(x_0)-f(x^*)\right\} \, .
\end{align*}
\end{theorem}

\begin{proof}
By assumption, there exists an $R \in \reals$, s.t. $\norm{x_k - x^*} \leq R$, for all $k= 0, 1, \dots, N$. With Corollary~\ref{cor:onestep_opt} it follows for any step size $h_k \geq 0$:
\begin{align}
\label{eq:unbounded}
\expect \left[ f(x_{k+1})-f(x^*) \mid x_k \right] %
&\leq \left(1-\frac{h_k}{n}\right)  \left(f(x_k)-f(x^*) \right)  + \frac{L_1 h_k^2}{2n} R^2 + \frac{L_1 \mu^2}{2} \, .
\end{align}
Taking expectation we obtain
\begin{align*}
\expect\left[f(x_{k+1})-f(x^*) \right] & \leq \left(1- \frac{h_k}{n}\right) \expect\left[f(x_k)-f(x^*) \right] +  \left(\frac{h_k}{n}\right)^2 \frac{nL_1 R^2}{2} + \frac{L_1 \mu^2}{2} \, . 
\end{align*}
By setting $h_k := \frac{2n}{k+1} \text{ for } k=0, \dots, (N-1)$ 
we obtain a recurrence that is exactly of the form as treated in Lemma~\ref{lemma:ind} and the result follows.
\end{proof}

We note that for $\epsilon > 0$ the exact algorithm $\mathcal{RP}_{0}$
needs $O\left(\frac{n}{\epsilon}\right)$ steps to guarantee an
approximation error of $\epsilon$. According to the discussion
preceding Lemma~\ref{lemma:onesteprel}, this still holds under an approximate
line search with fixed relative error. 

In the absolute error model, however, the error bound of
Theorem~\ref{thm:convex_bounded} becomes meaningless as
$N\rightarrow\infty$. Nevertheless, for $N_\text{opt} = \sqrt{2 Q/(L_1
  \mu^2)}$ the bound yields
\begin{align*}
\expect\left[f(x_{N_\text{opt}})-f(x^*) \right] \leq \frac{Q}{N_\text{opt}}  + \frac{N_\text{opt} L_1 \mu^2}{2} \leq \mu \sqrt{2 Q L_1}.
\end{align*}
%

\subsection{Remarks}
We emphasize that the constant $L_1$ and the strong-convexity parameter $m$ which describe the behavior of the function are only needed for the theoretical analysis of $\mathcal{RP}_{\mu}$. These parameters are \emph{not} input parameters to the algorithm. No pre-calculation or estimation of these parameters is thus needed in order to use the algorithm on convex functions. Moreover, the presented analysis does not need parameters that describe the properties of the function on the whole domain. It is sufficient to restrict our view on the sub-level set determined by the initial iterate. Consequently, if the function parameters get better in a neighborhood of the optimum, the performance of the algorithm may be better than theoretically predicted by the worst-case analysis.



\section{Computational Experiments}
\label{sec:experiments}
We complement the presented theoretical analysis with extensive numerical optimization experiments on selected benchmark functions. We compare the performance of the $\mathcal{RP}_{\mu}$ algorithm with a number of gradient-free algorithms that share the simplicity of Random Pursuit in terms of the computational search primitives used. We also introduce a heuristic acceleration scheme for Random Pursuit, the accelerated $\mathcal{RP}_{\mu}$ method ($\mathcal{ARP}_{\mu}$). As a generic reference we also consider two steepest descent schemes that use analytic gradient information. The test function set comprises two quadratic functions with different condition numbers, two variants of Nesterov's smooth function \cite{nesterov:book}, and a non-convex funnel-shaped function. We first detail the algorithms, their input requirements, and necessary parameter choices. We then present the definition of the test functions, describe the experimental performance evaluation protocol, and present the numerical results. Further experimental data are available in the supporting online material~\cite{Stich:2011s} at \texttt{http://arxiv.org/abs/1111.0194}.

\subsection{Algorithms}
We now introduce the set of tested algorithms. All methods have been implemented in MATLAB. The source code is also publicly available in the supporting online material~\cite{Stich:2011s}.

\subsubsection{Random Gradient Methods}
We consider two randomized methods that are discussed in detail in \cite{nesterov:randomgradientfree}. The first algorithm, the Random Gradient Method ($\mathcal{RG}$), implements the iterative scheme described in (\ref{eq:random2}). A necessary ingredient for the algorithm is an oracle that provides directional derivatives. The accuracy of the directional derivatives is controlled by the finite difference step size $\mu$. A pseudo-code representation of the approximate Random Gradient method ($\mathcal{RG}_\mu$) along with a convergence proof is described in~\cite[Section 5]{nesterov:randomgradientfree}. We implemented $\mathcal{RG}_\mu$ and used the parameter setting $\mu=1\textsc{e}-5$. A necessary input to the  $\mathcal{RG}_\mu$ algorithm is the function-dependent Lipschitz constant $L_1$ that is used to determine the step size $\lambda_k=1/(4 (n+4) L_1)$. We also consider Nesterov's fast Random Gradient Method ($\mathcal{FG}$) ~\cite{nesterov:randomgradientfree}. This algorithm simultaneously evolves two iterates in the search space where, in each iteration, a directional derivative is approximately computed at specific linear combinations of these points. In~\cite[Section 6]{nesterov:randomgradientfree} Nesterov provides a pseudo-code for the approximate scheme $\mathcal{FG}_\mu$ and proves convergence on strongly convex functions. We implemented the $\mathcal{FG}_\mu$ scheme and used the parameter setting $\mu=1\textsc{e}$-5. Further necessary input parameters are both the $L_1$ constant and the strong convexity parameter $m$ of the respective test function.

\subsubsection{Random Pursuit Methods}
In the implementation of the $\mathcal{RP}_{\mu}$ algorithm we choose the sampling directions uniformly at random from the hypersphere. We use the built-in MATLAB routine \texttt{fminunc.m} from the optimization toolbox \cite{MATLAB:2012} with \texttt{optimset('TolX'=$\mu$)} as approximate line search oracle $\lsapprox_{\mu}$ with $\mu =1\textsc{e}$-5. In the present gradient-free setting \texttt{fminunc.m} uses a mixed cubic/quadratic polynomial line search where the first three points bracketing the minimum are found by bisection \cite{MATLAB:2012}. 

Inspired by the $\mathcal{FG}$ scheme we also designed an accelerated Random Pursuit algorithm ($\mathcal{ARP}_{\mu}$) which is summarized in Algorithm~\ref{alg:arp}. 
\begin{algorithm}[h]
	\caption{Accelerated Random Pursuit ($\mathcal{ARP}_{\mu}$)}
	\label{alg:arp}
	  \begin{algorithmic}[1]
	  \REQUIRE $\begin{array}{l} %
                 N, x_0, \mu, m, L_1
          \end{array}$
\medskip
	    \STATE $\theta = \frac{1}{L_1 n^2}$, $\gamma_0 \geq m$, $\nu_0 = x_0$
			\FOR {$k = 0$ to $N$}
			\STATE Compute $\beta_k > 0$ satisfying $\theta^{-1} \beta_k^2 = (1-\beta_k) \gamma_k + \beta_k m =: \gamma_{k+1}$.
			\STATE Set $\lambda_k=\frac{\beta_k}{\gamma_{k+1}} m$, $\delta_k = \frac{\beta_k \gamma_k}{\gamma_k + \beta_k m}$, and $y_k=(1-\delta_k) x_k + \delta_k v_k$.
			\STATE $u_k \sim S^{n-1}$
			\STATE Set $x_{k+1} = y_k + \lsapprox_\mu(y_k,u_k) \cdot u_k$.
			\STATE Set $v_{k+1}=(1-\lambda_k)v_k + \lambda_k y_k + \frac{\lsapprox_\mu(y_k,u_k)}{\beta_k n} u_k$.
			\ENDFOR
		\end{algorithmic}
\end{algorithm}
The structure of this algorithm is similar to Nesterov's $\mathcal{FG}_\mu$ scheme. In $\mathcal{ARP}_{\mu}$ the step size calculation is, however, provided by the line search oracle. Although we currently lack theoretical guarantees for this scheme we here report the experimental performance results. Analogously to the $\mathcal{FG}_\mu$ algorithm, the accelerated $\mathcal{RP}_{\mu}$ algorithm needs the function-dependent parameters $L_1$ and $m$ as necessary input. The line search oracle is identical to the one in standard Random Pursuit. 

\subsubsection{Adaptive Step Size Random Search Methods} The previous randomized schemes proceed along random directions either by using pre-calculated step sizes or by using line search oracles. In adaptive step size random search methods the step size is dynamically controlled such as to approximately guarantee a certain probability $p$ of finding an improving iterate. Schumer and Steiglitz~\cite{schumer:68} were among the first to propose such a scheme. In the bio-inspired optimization literature, the method is known as the (1+1)-Evolution Strategy ($\mathcal{ES}$) \cite{Rechenberg:1973}. J\"agersk\"upper \cite{jaeg:05} provides a convergence proof of $\mathcal{ES}$ on convex quadratic functions. We here consider the following generic $\mathcal{ES}$ algorithm summarized in Algorithm~\ref{alg:es}.
\begin{algorithm}[htb]
	\caption{(1+1)-Evolution Strategy ($\mathcal{ES}$) with adaptive step size control}
	\label{alg:es}
	  \begin{algorithmic}[1]
	  	 \REQUIRE $\begin{array}{l} %
                 N, x_0, \sigma_0, \text{probability of improvement $p=0.27$}
          \end{array}$
\medskip
      \STATE Set $c_s = e^{\frac{1}{3}}, c_f= c_s \cdot e^{\frac{-p}{1-p}}$.
			\FOR {$k = 0$ to $N$}
			  \STATE $u_k \sim \normal(0, I_n)$
			  \IF {$f(x_k + \sigma_k u_k) \leq f(x_k)$} 
			  \STATE Set $x_{k+1}=x_k + \sigma_k u_k$ and $\sigma_{k+1}=c_s\cdot \sigma_k$. 
			  \ELSE  
			  \STATE Set $x_{k+1}=x_k$ and $\sigma_{k+1}=c_f \cdot \sigma_k$.
			  \ENDIF
			\ENDFOR
		\end{algorithmic}
\end{algorithm}\\
Depending on the specific random direction generator and the underlying test function different optimality conditions can be formulated for the probability $p$. Schumer and Steiglitz \cite{schumer:68} suggest the setting $p=0.27$ which is also considered in this work. For all of the considered test functions the initial step size $\sigma_0$ has been determined experimentally in order to guarantee the targeted $p$ at the start (see Table~\ref{tab:sigma} for the respective values). The $\mathcal{ES}$ algorithm shares $\mathcal{RP}_{\mu}$'s invariance under strictly monotone transformations of the objective function.

\subsubsection{First-order Gradient Methods}
In order to illustrate the numerical efficiency of the randomized zeroth-order schemes relative to that of first-order methods, we also consider two Gradient Methods as outlined in (\ref{eq:GM}). The first method ($\mathcal{GM}$) uses a fixed step size $\lambda_k=\frac{1}{L_1}$~\cite{nesterov:book}. The function-dependent constant $L_1$ is, thus, part of the input to the $\mathcal{GM}$ algorithm. The second method ($\mathcal{GM}_{LS}$) determines the step size $\lambda_k$ in each iteration using $\mathcal{RP}_{\mu}$ line search oracle $\lsapprox_{\mu}$ with $\mu=1\textsc{e}$-5 as input parameter. 

\subsection{Benchmark Functions}
We now present the set of test functions used for the numerical performance evaluation of the different optimization schemes. We present the three function classes and detail the specific function instances and their properties. 

\subsubsection{Quadratic Functions}
We consider quadratic test functions of the form:
\begin{align}
\label{fn:quadratic}
f(x) &= \frac{1}{2} (x-1)^T Q (x-1) \, ,
\end{align}
where $x \in \reals^n$ and $Q \in \reals^{n \times n}$ is a diagonal matrix. For given $L_1$ the diagonal entries are chosen in the interval $[1,L_1]$. The minimizer of this function class is $x^* = 1$ and $f(x^*) = 0$. The derivative is $\nabla f(x) = Q(x-1)$.
We consider two different matrix instances. Setting $Q=I_n$ the $n$-dimensional identity matrix the function reduces to the shifted sphere function denoted here by $f_1$. In order to get a quadratic function with anisotropic axis lengths we use a matrix $Q$ whose first $n/2$ diagonal entries are equal to $L_1$ and the remaining entries are set to $1$. This ellipsoidal function is denoted by $f_2$.

\subsubsection{Nesterov's Smooth Functions}
We consider Nesterov's smooth function as introduced in Nesterov's text book~\cite{nesterov:book}. The generic version of this function reads:
\begin{align}
\label{fn:nesterovsmooth}
f_3(x) = \frac{L_1}{4} \left( \frac{1}{2} \left[ x_1^2 + \sum_{i=1}^{n-1} \left(x_{i+1}-x_i \right)^2 + x_n^2 \right] - x_1 \right) \, .
\end{align}
This function has derivative $\nabla f_3(x) =\frac{L_1}{4} \left(Ax - e_1\right)$, where
\begin{align*}
A &= \begin{bmatrix} 
      2 & -1 & 0 &0   &  &  &  & \\
     -1 & 2  & -1 &0 & & & 0 & \\
     0 & -1 & 2 &1   &  & &  & \\ 
      & & \ddots & \ddots &  \ddots  & &  &  \\
      &  & 0 &   &  &  -1 &  2 & -1\\
      &  &  &   &  & 0 &  -1 & 2\\
     \end{bmatrix} %
& &\text{and} & e_1 &= (1, 0, \dots, 0)^T.
\end{align*}
The optimal solution is located at:
\begin{align*}
x^*_i = 1- \frac{i}{n+1} \, , \quad \text{for } i=1, \dots, n.
\end{align*}
For fixed dimension $n$, this function is strongly convex with parameter $m \approx \frac{L_1}{4(n+1)^2}$. Thus, the condition $L_1/m$ grows quadratically with the dimension.
Adding a regularization term leads, however, to a strongly convex function with parameter $m$ independent of the dimension. Given $L_1 \geq m > 0$, the regularized function reads:
\begin{align}
\label{fn:nesterovstrong}
f_4(x) = \frac{L_1-m}{4} \left( \frac{1}{2} \left[ x_1^2 + \sum_{i=1}^{n-1} \left(x_{i+1}-x_i \right)^2 + x_n^2 \right] - x_1 \right) + \frac{m}{2} \norm{x}^2 \, .
\end{align}
This function is strongly convex with parameter $m$.\\
Its derivative $\nabla f_4(x) =\left( \frac{L_1-m}{4} A + m I \right) x - \frac{L_1-m}{4} e_1$, and the optimal solution $x^*$ satisfies $\left(A + \frac{4m}{L_1-m} \right) x^* = e_1$.

\subsubsection{Funnel Function}\label{sec:funnel}
We finally consider the following funnel-shaped function 
\begin{align}
\label{fn:funnel}
f_5(x) = \log \left( 1 + 10 \sqrt{(x-1)^T (x-1)} \right) \, ,
\end{align}
where $x \in \reals^n$. The minimizer of this function is $x^* = 1$ with $f_5(x^*) = 0$. Its derivative for $x\neq 1$ is $\nabla f_5(x) = 10 /(1+10  \sqrt{(x-1)^T (x-1)} ) \cdot \sign{(x-1)}$. A one-dimensional graph of $f_5$ is shown in the left panel of Figure~\ref{fig:figure6}. The function $f_5$ arises from a strictly monotone transformation of $f_1$ and thus belongs to the class of strictly quasiconvex functions.

\subsection{Numerical Optimization Results}
To illustrate the performance of Random Pursuit in comparison with the other randomized methods we here present and discuss the key numerical results. For all numerical tests we follow the identical protocol. All algorithms use as starting point $x_0 = 0$ for all test functions. In order to compare the performance of the different algorithms across different test functions, we follow Nesterov's approach \cite{nesterov:randomgradientfree} and report relative solution accuracies with respect to the scale $S \approx \frac{1}{2} L_1 R^2$ where $R := \norm{x_0 - x^*}$ is the Euclidean distance between starting point and optimal solution of the respective function. The properties of the four convex and continuously differentiable test functions and the quasiconvex funnel function along with the upper bounds on $R^2$ and the corresponding scales $S$ are summarized in Table~\ref{tab:testfunctions}.         
\begin{table}[H]
	\centering
	\scalebox{0.95}{
	\begin{tabular}{|>{$}l<{$} |l|l|l|l|l|>{$}l<{$}|}
	\hline
	   & \textsc{Name} & function class & $L_1$ & $m$ & $R^2$ & $S$ \\ \hline
		f_1 & \textsc{Sphere} & strongly convex & 1 & 1  & n & \frac{1}{2} n\\ \hline
		f_2 & \textsc{Ellipsoid} & strongly convex& 1000 & 1  & n & 50 n\\ \hline
		f_3 & \textsc{Nesterov smooth}  & convex & 1000 & $\approx \frac{L_1}{4(n+1)^2}$  & $\frac{n+1}{3}$ & 500 \cdot \frac{n+1}{3}\\ \hline
		f_4 & \textsc{Nesterov strong} & strongly convex & 1000 & 1  & $\frac{\sqrt{1000}}{4}$ & 1000 \\ \hline
		f_5 & \textsc{Funnel} & not convex & - & - & n & \frac{1}{2} n \\ \hline
	\end{tabular}
	}
	\caption{Test functions with parameters $L_1$, $m$, $R$ and the used scale $S$.}
	\label{tab:testfunctions}
\end{table}
Due to the inherent randomness of a single search run we perform 25 runs for each pair of problem instance/algorithm with different random number seeds. We compare the different methods based on two record values: (i) the minimal, mean, and maximum number of \emph{iterations} (ITS) and (ii) the minimal, mean, and maximum number of  \emph{function evaluations} (FES) needed to reach a certain solution accuracy. While the former records serve as a means to compare the number of oracle calls in the different method, the latter one only considers evaluations of the objective function as relevant performance cost. It is evident that measuring the performance of the algorithms in terms of oracle calls favors Random Pursuit because the line search oracle ``does more work" than an oracle that, for instance, provides a directional derivative. For Random Gradient methods the number of FES is just twice the number of ITS when a first-order finite difference scheme is used for directional derivatives. For the $\mathcal{ES}$ algorithm the number of ITS and FES is identical. For Random Pursuit methods the relation between ITS and FES depends on the specific test function, the line search parameter $\mu$, and the actual implementation of the line search. 
Our theoretical investigation suggest that the randomized schemes are a factor of $O(n)$ times slower than the first-order $\mathcal{GM}$ algorithms. This is due the reduced available (direction) information in the randomized methods compared to the $n$-dimensional gradient vector. For better comparison with $\mathcal{GM}$ and $\mathcal{GM}_{LS}$, we thus scale the number of ITS of the randomized schemes by a factor of $1/n$.

\subsubsection{Performance on the Quadratic Test Functions for $\mathbf{n\le1024}$}
We first consider the two quadratic test functions in $n=2^2,\ldots,2^{10}$ dimensions. Table~\ref{tab:sphere2} summarizes the minimum, maximum, and mean number of ITS (in blocks of size $n$) needed for each algorithm to reach the absolute accuracy $1.91\cdot10^{-6}S$ on the sphere function $f_1$. For the first-order $\mathcal{GM}$ algorithms the absolute number of ITS is reported.  
\begin{table}[ht]
	\centering
	\scalebox{0.62}{
	\begin{tabular}{ | r | r r r | r r r | r r r |  r r r |  r r r | r | r |}
  \hline
  \multicolumn{1}{|r|}{ } & \multicolumn{3}{c|}{$\mathcal{RP}$}  & \multicolumn{3}{c|}{$\mathcal{RG}$} 
						  & \multicolumn{3}{c|}{$\mathcal{FG}$}
                          & \multicolumn{3}{c|}{$\mathcal{ARP}$} & \multicolumn{3}{c|}{$\mathcal{ES}$}
                          & \multicolumn{1}{c|}{$\mathcal{GM}$} 
                          & \multicolumn{1}{c|}{$\mathcal{GM}_{LS}$}   \\
  n  &  min & max & mean & min & max & mean & min & max & mean & min & max & mean & min & max & mean & - & - \\
  \hline 
4   &   5 & 17 & 10 &   40 & 65 & 53  &  31 & 49 & 39 &    5 & 17 & 10 &  28 & 46 & 38 &      1 &   1\\
8   &   8 & 16 & 12 &   39 & 53 & 44  &  30 & 40 & 35 &    5 & 13 & 11 &  28 & 43 & 35 &      1 &   1\\
16  &  10 & 14 & 12 &   33 & 41 & 37  &  30 & 37 & 33 &   10 & 14 & 12 &  30 & 42 & 36 &      1 &   1\\
32  &  11 & 14 & 12 &   31 & 36 & 33  &  28 & 35 & 31 &   11 & 16 & 12 &  33 & 41 & 37 &      1 &   1\\
64  &  12 & 14 & 13 &   30 & 34 & 32  &  28 & 33 & 31 &   12 & 14 & 13 &  33 & 41 & 37 &      1 &   1\\
128 &  12 & 14 & 13 &   30 & 32 & 31  &  29 & 32 & 31 &   12 & 14 & 13 &  35 & 40 & 37 &      1 &   1\\
256 &  13 & 14 & 13 &   30 & 31 & 30  &  29 & 31 & 30 &   13 & 14 & 13 &  35 & 40 & 37 &      1 &   1\\
512 &  13 & 13 & 13 &   30 & 31 & 30  &  30 & 31 & 30 &   13 & 14 & 13 &  36 & 38 & 37 &      1 &   1\\
1024&  13 & 14 & 13 &   30 & 31 & 30  &  30 & 31 & 30 &   13 & 13 & 13 &  36 & 38 & 37 &      1 &   1\\
  \hline  
\end{tabular}	

	}
  \caption{Recorded minimum, maximum, and mean \#ITS/$n$ on the sphere function $f_1$ to reach a relative accuracy of $1.91 \cdot 10^{-6}$. For $\mathcal{GM}$ and $\mathcal{GM}_{LS}$ the absolute number of ITS are recorded.}
	\label{tab:sphere2}
\end{table}
Three key observations can be made from these data. First, all zeroth-order algorithms approach the theoretically expected \emph{linear scaling of the run time with dimension} for strongly convex functions for sufficiently large $n$ (for $n\ge64$, e.g., the average number of ITS/n becomes constant for the $\mathcal{RP}$ algorithms). 
Second, no significant performance difference can be found between $\mathcal{RP}$ and its accelerated version across all dimensions. The performance of the algorithm pair $\mathcal{RG} / \mathcal{FG}$ becomes similar for $n\ge128$. 
Third, the Random Pursuit algorithms outperform all other zeroth-order methods in terms of number of ITS. Only the last observation changes when the number of FES is considered. Table~\ref{tab:sphere2_f} summarizes the number of FES (in blocks of size $n$) for all algorithms on $f_1$.
\begin{table}[ht]
	\centering
	\scalebox{0.62}{
	\begin{tabular}{ | r | r r r | r r r | r r r |  r r r |  r r r |}
  \hline
  \multicolumn{1}{|r|}{ } & \multicolumn{3}{c|}{$\mathcal{RP}$}  & \multicolumn{3}{c|}{$\mathcal{RG}$} 
                          & \multicolumn{3}{c|}{$\mathcal{FG}$}
                          & \multicolumn{3}{c|}{$\mathcal{ARP}$} & \multicolumn{3}{c|}{$\mathcal{ES}$}	\\
  n  &  min & max & mean & min & max & mean & min & max & mean & min & max & mean & min & max & mean\\
  \hline 
4   & 20 & 69 & 39 &   80 & 131 & 106 &   62 & 99 & 78 &   20 & 69 & 39 &   28 & 46 & 38 \\
8   & 34 & 65 & 47 &   78 & 105 &  87 &   59 & 81 & 70 &   22 & 53 & 43 &   28 & 43 & 35 \\
16  & 38 & 54 & 48 &   65 &  81 &  73 &   61 & 74 & 66 &   40 & 56 & 47 &   30 & 42 & 36 \\
32  & 45 & 57 & 50 &   62 &  71 &  66 &   57 & 69 & 62 &   43 & 62 & 50 &   33 & 41 & 37 \\
64  & 47 & 57 & 52 &   60 &  68 &  64 &   57 & 66 & 61 &   50 & 56 & 52 &   33 & 41 & 37 \\
128 & 50 & 56 & 53 &   59 &  64 &  62 &   58 & 64 & 61 &   51 & 56 & 54 &   35 & 40 & 37 \\
256 & 56 & 63 & 59 &   59 &  62 &  61 &   58 & 62 & 60 &   56 & 63 & 59 &   35 & 40 & 37 \\
512 & 64 & 69 & 67 &   59 &  62 &  60 &   59 & 62 & 60 &   64 & 70 & 67 &   36 & 38 & 37 \\
1024& 84 & 92 & 89 &   59 &  61 &  60 &   60 & 61 & 60 &   85 & 91 & 88 &   36 & 38 & 37 \\
  \hline  
\end{tabular}	

	}
  \caption{Recorded minimum, maximum, and mean \#FES/$n$ on the sphere function $f_1$ to reach a relative accuracy of $1.91 \cdot 10^{-6}$.}
	\label{tab:sphere2_f}
\end{table}
We see that the $\mathcal{RP}_{\mu}$ algorithms outperform the Random Gradient methods for low dimensions and perform equally well for $n=256$. For $n\ge512$ the Random Gradient schemes become increasingly superior to the Random Pursuit schemes. Remarkably, the adaptive step size $\mathcal{ES}$ algorithm outperforms all other methods across all dimensions. The data also reveal that the line search oracle in the $\mathcal{RP}_{\mu}$ algorithms consume on average \emph{four FES per iteration} for $n\le128$ with a slight increase to \emph{seven FES per iteration} for $n=1024$. We also observe that the gap between minimum and maximum number of FES reduces with increasing dimension for all methods. Finally, the first-order schemes reach the minimum as expected in a single iteration across all dimension. 

For the high-conditioned ellipsoidal function $f_2$ we observe a genuinely different behavior of the different algorithms. Figure~\ref{fig:figure1} shows for each algorithm the mean number of FES (left panel) and ITS (right panel) in blocks of size $n$ needed to reach the absolute accuracy $1.91\cdot10^{-6}S$ on $f_2$. The minimum, maximum, and mean number of ITS and FES are reported in the Appendix in Tables~\ref{tab:ellipsoid2} and \ref{tab:ellipsoid2_f}, respectively.
\begin{figure}[ht]
\begin{minipage}[b]{0.49\linewidth}
\centering
\includegraphics[scale=.51]{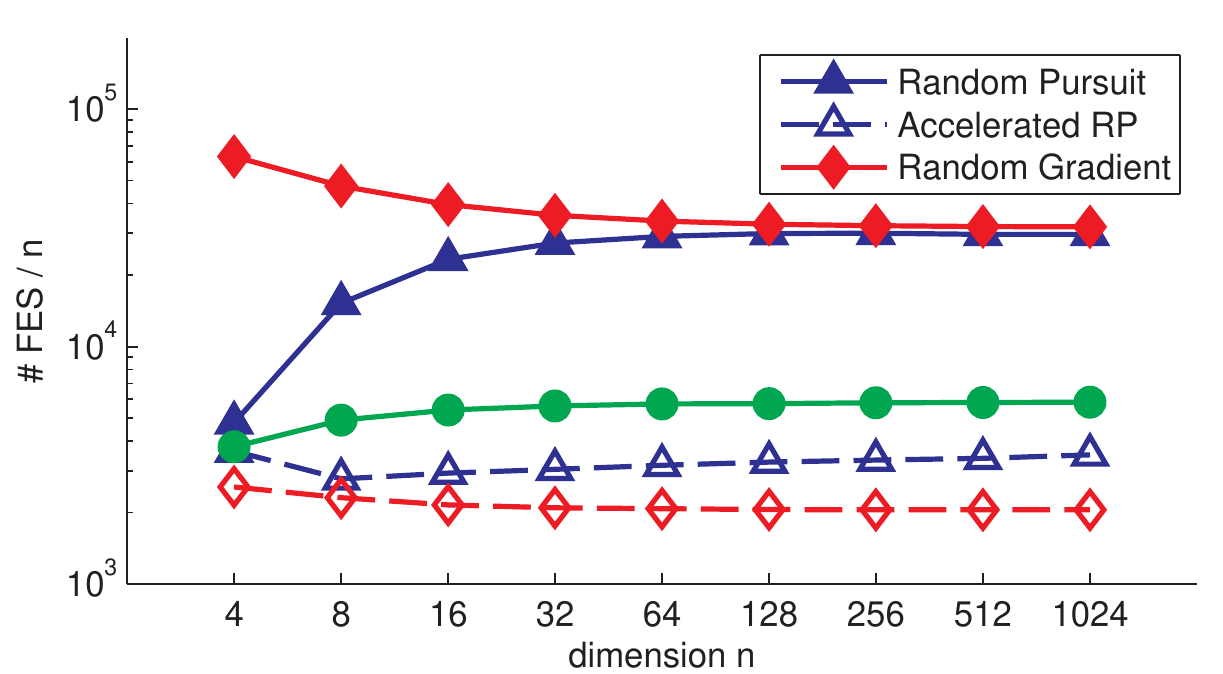}
\end{minipage}
\hspace{0.1cm}
\begin{minipage}[b]{0.49\linewidth}
\centering
\includegraphics[scale=.51]{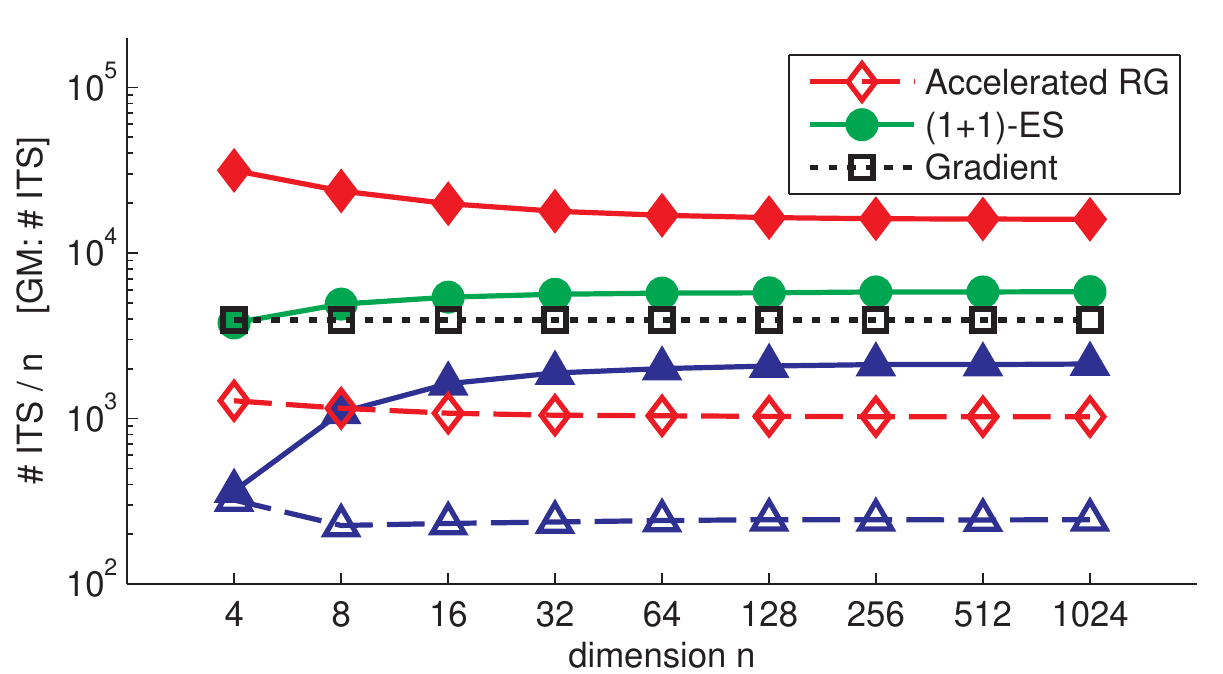}
\end{minipage}
\caption{Average \#FES/$n$ (left panel) and \#ITS/$n$ (right panel) vs. dimension $n$ on the ellipsoidal function $f_2$ to reach a relative accuracy of $1.91 \cdot 10^{-6}$ (\#ITS for $\mathcal{GM}$). Further data are available in Tables~\ref{tab:ellipsoid2} and~\ref{tab:ellipsoid2_f}.}
\label{fig:figure1}
\end{figure}

%
We again observe the theoretically expected linear scaling of the number of ITS with dimension for sufficiently large $n$. The mean number of ITS now spans two orders of magnitude for the different algorithms. Standard Random Pursuit outperforms the $\mathcal{RG}$ and the $\mathcal{ES}$ algorithm. 
%
%
Moreover, the accelerated $\mathcal{RP}_{\mu}$ scheme outperforms the $\mathcal{FG}$ scheme by a factor of 4. All methods show, however, an increased overall run time due to the high condition number of the quadratic form. This is also reflected in the increased number of FES that are needed by the line search oracle in the $\mathcal{RP}_{\mu}$ algorithms. The line search oracle now consumes on average \emph{12-14 FES per iteration}. 

In terms of consumed FES budget we observe that Random Pursuit still outperforms Random Gradient for small dimensions but needs a comparable number of FES for $n\ge64$ (around 30.000 FES in blocks of $n$). The $\mathcal{ES}$, the $\mathcal{ARP}_{\mu}$, and the $\mathcal{FG}$ algorithm need an order of magnitude fewer FES. The accelerated $\mathcal{RP}_{\mu}$ is only outperformed by the $\mathcal{FG}$ algorithm. The performance of the $\mathcal{ES}$ algorithm is again remarkable given the fact that it does not need information about the parameters $L_1$ and $m$ which are of fundamental importance for the accelerated schemes. 

\subsubsection{Performance on the Full Benchmark Set for $\mathbf{n=64}$} We now illustrate the behavior of the different algorithms on the full benchmark set for fixed dimension $n=64$. We observed similar qualitative behavior for all other dimensions. Table~\ref{tab:sum1} contains  the number of ITS needed to reach the scale-dependent accuracy $1.91 \cdot 10^{-6}S$ for all algorithms.
\begin{table}[ht]
	\centering
	\scalebox{0.60}{
	\begin{tabular}{ | r | r r r | r r r | r r r |  r r r |  r r r | r | r |}
  \hline
  \multicolumn{1}{|r|}{ } & \multicolumn{3}{c|}{$\mathcal{RP}$}  & \multicolumn{3}{c|}{$\mathcal{RG}$} 
						  & \multicolumn{3}{c|}{$\mathcal{FG}$}
                          & \multicolumn{3}{c|}{$\mathcal{ARP}$} & \multicolumn{3}{c|}{$\mathcal{ES}$}	
                          & \multicolumn{1}{c|}{$\mathcal{GM}$} 
                          & \multicolumn{1}{c|}{$\mathcal{GM}_{LS}$}  \\  
  fun.  &  min & max & mean & min & max & mean  &  min & max & mean & min & max & mean & min & max & mean & - & - \\						  
\hline 
$f_1$&   \textbf{12}&   \textbf{14}&  \textbf{13}&     30&     34&     32&    30& 34  & 32  &   \textbf{12}&  \textbf{14} &  \textbf{13} & 33   & 41  & 37  &      1&  1   \\
$f_2$&          1899&         2096 &        2001 &  16601&  17333&  16868&   990& 1079& 1038&  \textbf{233}&  \textbf{250}&  \textbf{242}&  5451& 5954& 5729&   3934&  3   \\
$f_3$&          2068&         2191 &       2136  &  18922&  19075&  19004&   892& 970 & 942 &  \textbf{192}& \textbf{678} & \textbf{473} & 5766 & 6050& 5916&  4474 &  2237\\
$f_4$&          954 &        1023  &         995 &   8727&   8995&   8854&   441& 534 & 458 &  \textbf{137}&  \textbf{188}&  \textbf{159}&  2651& 2854& 2751&  2086 &  1044 \\
$f_5$&  \textbf{26} & \textbf{30}  & \textbf{28} &  -    &  -    &  -    &   -  & -   & -   & \textbf{26}  & \textbf{30}  & \textbf{28}  & 73   & 85  & 78  & -     &  1    \\
\hline  
\end{tabular}	

	}
  \caption{Average \#ITS/$n$ to reach the relative accuracy $1.91 \cdot 10^{-6}$ in dimension $n=64$. For $\mathcal{GM}$ and $\mathcal{GM}_{LS}$ the exact number of ITS is reported. Observed minimum ITS across all (gradient-free) algorithms are marked in bold face for each function.}
	\label{tab:sum1}
\end{table}
We observe that Random Pursuit outperforms the $\mathcal{RG}$ and the $\mathcal{ES}$ algorithm, and that the $\mathcal{ARP}_{\mu}$ algorithm outperforms all gradient-free schemes in terms of number of ITS on all functions (with equal performance as Random Pursuit on $f_1$ and $f_5$). 
We consistently observe an improved performance of all algorithms on the regularized strongly convex function $f_4$ as compared to its convex counterpart $f_3$. This expected behavior is most pronounced for the $\mathcal{ARP}_{\mu}$ scheme where, on average, the number of ITS is reduced to $159/473 \approx 1/3$. 

A comparison between the two gradient schemes reveals that $\mathcal{GM}_{LS}$ outperforms the fixed step size gradient scheme on all test functions. The remarkable performance of $\mathcal{GM}_{LS}$ on $f_2$ is due to the fact that the spectrum of the Hessian contains in equal parts two different values (1 and $L$,  respectively). A \emph{single} line search along a gradient direction is thus simultaneously optimal for $n/2$ directions of this function. The $\mathcal{GM}_{LS}$ scheme thus reaches the target accuracy in as few as three steps. This efficiency is lost as soon as the spectrum becomes sufficiently diverse (as indicated by its performance on $f_4$). We also remark that the performance of $\mathcal{RP}/\mathcal{GM}_{LS}$ as well as the pair $\mathcal{FG}/\mathcal{GM}$ is in full agreement with theory. We see on functions $f_3/f_4$ that $\mathcal{RP}$ is about a factor of $n$ slower than the $\mathcal{GM}_{LS}$ due to unavailable gradient information. The same is true for $\mathcal{FG}/\mathcal{GM}$ where $\mathcal{FG}$ is about $4 n$ times slower than $\mathcal{GM}$ due to the theoretically needed reduction of the optimal step length by a factor of $1/4$ \cite{nesterov:randomgradientfree}.

For function $f_4$ we illustrate the convergence behavior of the different algorithms in Figure~\ref{fig:figure3}. After a short algorithm-dependent initial phase we observe linear convergence of all algorithms for fixed dimension, i.e., a constant reduction of the logarithm of the distance to the minimum per iteration.    
\begin{figure}[ht]
\centering
\includegraphics[scale=.8]{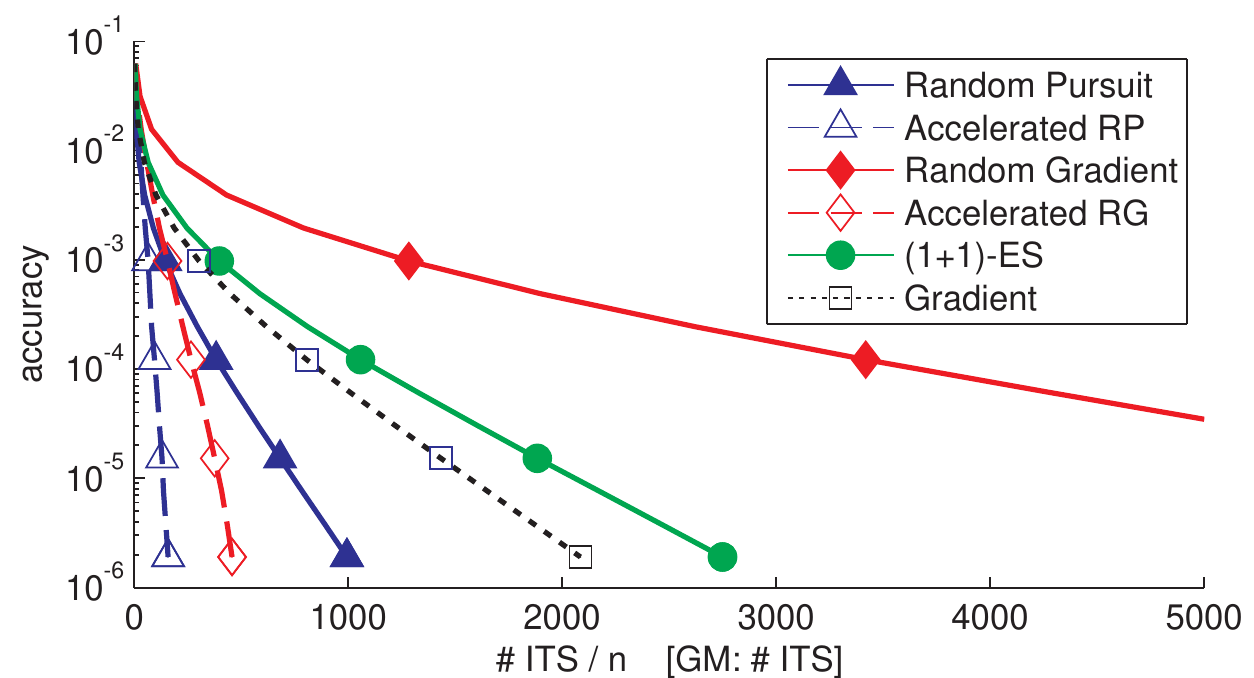}
\caption{Average accuracy (in log scale) vs. \#ITS/$n$ for all algorithms on $f_4$ in dimension $n=64$. Further data are available in~\cite{Stich:2011s} in Table S-4.}
\label{fig:figure3}
\end{figure}
We also observe that the accelerated Random Pursuit consistently outperforms standard Random Pursuit for all measured accuracies on $f_4$ (see Table~S-4 in~\cite{Stich:2011s} for the corresponding numerical data). This behavior is less pronounced for the function pair $f_1/f_2$ as shown in Figure~\ref{fig:figure4}. 
\begin{figure}[ht]
\centering
\includegraphics[scale=.7]{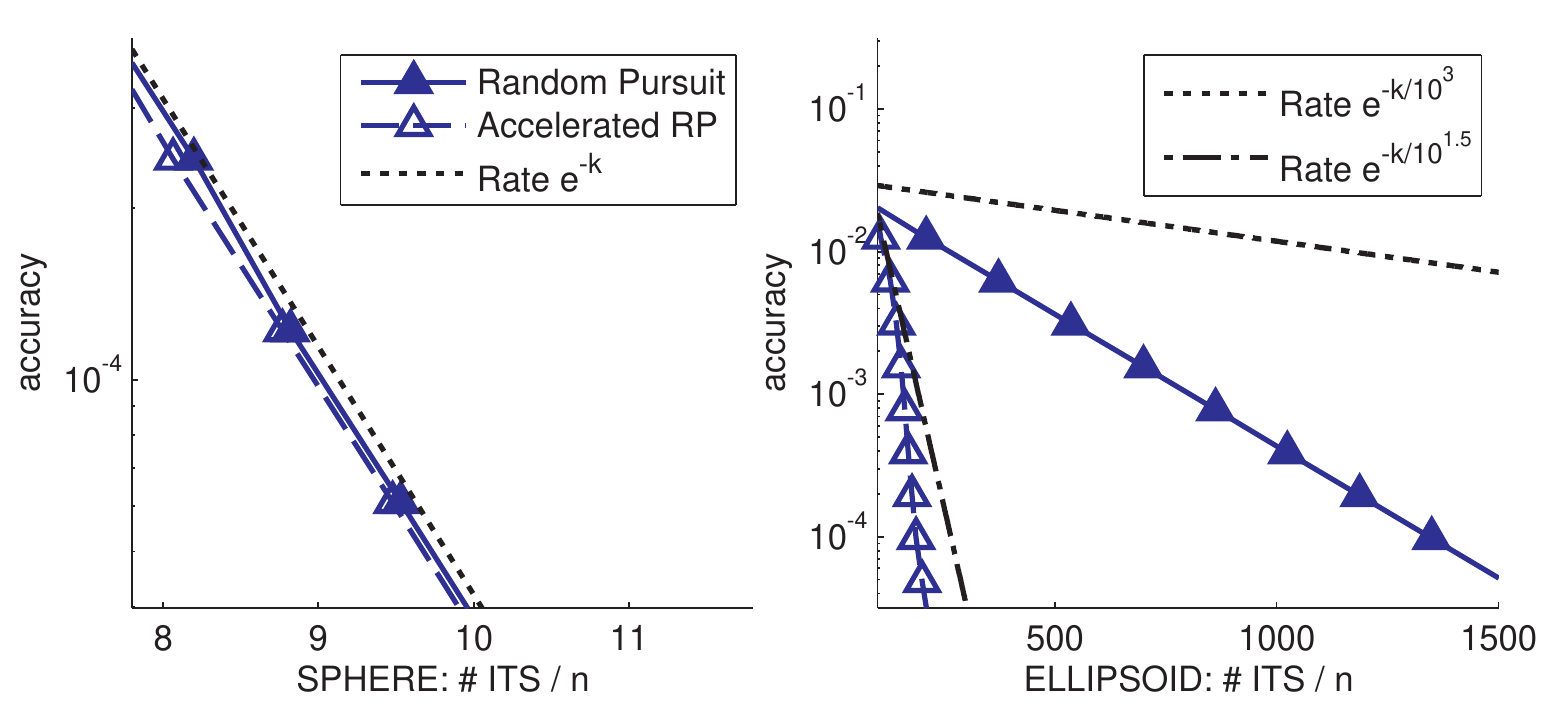}
\caption{Numerical convergence rate of standard and accelerated Random Pursuit on $f_1$ (left panel) and $f_2$ (right panel) in dimension $n=64$. For both instances the theoretically predicted worst-case progress rate (dotted line) is shown for comparison. The rate of the $\mathcal{ARP}$ scheme is compared to the theoretically predicted convergence rate of $\mathcal{FG}$ (slash-dotted line) \cite{nesterov:randomgradientfree}.}
\label{fig:figure4}
\end{figure}
On $f_1$ both Random Pursuit schemes have identical progress rates that are also consistent with the theoretically predicted one. On $f_2$ Random Pursuit  outperforms the accelerated scheme for low accuracies (see also Table~S-2 in~\cite{Stich:2011s} for the numerical data) but is quickly outperformed due to faster progress rate of the accelerated scheme. We also observe that the theoretically predicted worst-case progress rate (dotted line in the right panel of Figure~\ref{fig:figure4}) does not reflect the true progress on this test function. Comparison of the numerical results on the function pair $f_1/f_5$ (see Figure~\ref{fig:figure5}) demonstrates the expected invariance under strictly monotone transformations of the Random Pursuit algorithms and the $\mathcal{ES}$ scheme. 
These algorithms enjoy the same convergence behavior (up to small random variations) to the target solution while the Random Gradient schemes fail to converge to the target accuracy. Note, however, the numbers reported in e.g. Table~\ref{tab:sum1} are not identical for $f_1$ and $f_5$ because the used stopping criteria are dependent on the scale of the \emph{function values}. The convergence rates are the same, but more iterations are needed for $f_5$ because the required accuracy is considerably smaller.

\begin{figure}[ht]
\centering
\includegraphics[scale=.8]{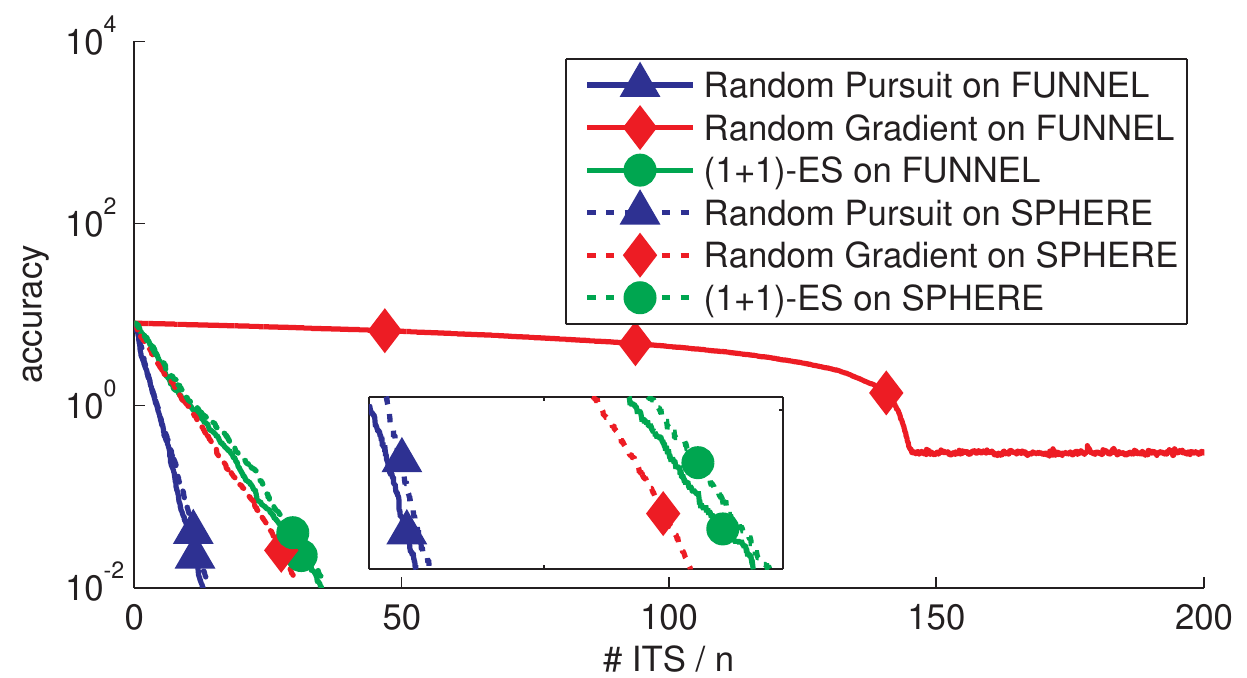}
\caption{Numerical convergence rate of the $\mathcal{RP}_{\mu}$, the $\mathcal{ES}$, and the $\mathcal{RG}$ scheme on $f_1$ and $f_5$ in $n=64$ dimensions. The accuracy is measured in terms of the logarithmic distance to the optimum $\log\left(\norm{x_k-x^*}_2\right)$.}
\label{fig:figure5}
\end{figure}

We also report the performance of the different algorithms in terms of number of FES needed to reach the target accuracy of  $1.91 \cdot 10^{-6}S$ for the different test functions. For all algorithms the minimum, maximum, and average number of FES are recorded in Table~\ref{tab:sum2}.  
\begin{table}[ht]
	\centering
	\scalebox{0.62}{
	\begin{tabular}{ | r | r r r | r r r |  r r r |  r r r |  r r r |}
  \hline
  \multicolumn{1}{|r|}{ } & \multicolumn{3}{c|}{$\mathcal{RP}$}  & \multicolumn{3}{c|}{$\mathcal{RG}$} 
                          & \multicolumn{3}{c|}{$\mathcal{FG}$}
                          & \multicolumn{3}{c|}{$\mathcal{ARP}$} & \multicolumn{3}{c|}{$\mathcal{ES}$}	\\
  fun. &  min & max & mean & min & max & mean & min & max & mean & min & max & mean & min & max & mean\\
  \hline 
$f_1$&  47   & 57   & 52   &   60   & 68   & 64   &   57           & 66           & 61           &   50  & 56  & 52  &   \textbf{33}& \textbf{41}& \textbf{37} \\
$f_2$& 27723 & 30272& 29071&   33202& 34667& 33736&   \textbf{1980}& \textbf{2159}& \textbf{2077}&   3035& 3247& 3159&   5451       & 5954       & 5729 \\
$f_3$&  25520& 27034& 26351&   37844& 38150& 38008&   \textbf{1785}& \textbf{1939}& \textbf{1885}&   2199& 8149& 5609&   5766       & 6050       & 5916 \\
$f_4$&  11629& 12482& 12122&   17455& 17990& 17708&   \textbf{883} & \textbf{1069}& \textbf{916} &   1557& 2134& 1825&   2651       & 2854       & 2751  \\
$f_5$& 338   & 384  & 360  &   -    & -    & -    &   -            & -            & -            &   342 & 399 & 361 &   \textbf{73}& \textbf{85}& \textbf{78} \\
  \hline  
\end{tabular}	

	}
  \caption{Average \#FES$/n$ to reach the relative accuracy $1.91 \cdot 10^{-6}$ in dimension $n=64$. Observed minimum \#FES/$n$ across all algorithms are marked in bold face for each function.}
	\label{tab:sum2}
\end{table}
We observe that the $\mathcal{RP}_{\mu}$ algorithm outperforms the standard Random Gradient method on all tested functions. However, Random Pursuit is not competitive compared to the accelerated schemes and the $\mathcal{ES}$ algorithm. The accelerated $\mathcal{RP}_{\mu}$ scheme is only outperformed by the $\mathcal{FG}$ algorithm. The latter scheme shows particularly good performance on the convex function $f_3$ with considerably lower variance. For functions $f_2$--$f_4$ the $\mathcal{RP}_{\mu}$ algorithms need around $12-15$ FES per line search oracle call. We emphasize again that the performance of the adaptive step size $\mathcal{ES}$ scheme is remarkable given the fact that it does not need any function-specific parametrization. A comparison to the parameter-free Random Pursuit scheme shows that it needs around four times fewer FES on functions $f_2$--$f_4$. 

We remark that Random Pursuit with discrete sampling, i.e., using the set of signed unit vectors for sample generation (see Section \ref{subsec:discSampl}), yields numerical results on the present benchmark that are consistent with our theoretical predictions. We observed improved performance of Random Pursuit with discrete sampling on the function triplet $f_1/f_2/f_5$. This is evident as the coordinate system of these functions coincide with the standard basis. Thus, algorithms that move along the standard coordinate system are favored. On the function pair $f_3/f_4$ we do not see any significant deviation from the presented performance results.
\begin{table}[ht]
	\centering
	\scalebox{0.62}{
	\begin{tabular}{ | r | r | r | r | r | r | r | r | r | r | r |}
  \hline
  \multicolumn{1}{|r|}{ param. $\mu$: } & \multicolumn{1}{c|}{1\textsc{e}$-1$} 
                          & \multicolumn{1}{c|}{1\textsc{e}$-2$} 
                          & \multicolumn{1}{c|}{1\textsc{e}$-3$} 
                          & \multicolumn{1}{c|}{1\textsc{e}$-4$} 
                          & \multicolumn{1}{c|}{1\textsc{e}$-5$} 
                          & \multicolumn{1}{c|}{1\textsc{e}$-6$} 
                          & \multicolumn{1}{c|}{1\textsc{e}$-7$} 
                          & \multicolumn{1}{c|}{1\textsc{e}$-8$} 
                          & \multicolumn{1}{c|}{1\textsc{e}$-9$} 						  
                          & \multicolumn{1}{c|}{1\textsc{e}$-10$} \\
  \hline 
\#ITS / n &  1986  &  1982  & 2013 &  2017  & 2001 & 2001 & 2001   &  2001  & 2020  & 2020   \\
\# FES / n &  19824  &  19781  &  21435  &  26120 & 29071 &  29495 & 29537   &  29542  &  38993  & 39070   \\
  \hline  
\end{tabular}	

	}
  \caption{Performance of $\mathcal{RP}_\mu$ on the ellipsoid function $f_2$, $m=1$, $L=1000$, $S=3200$, $n=64$, for different line search parameters $\mu$. Mean (of 25 repetitions) number of ITS/$n$ and FES/$n$ to reach a relative accuracy of $1.91 \cdot 10^{-6}$ are reported.}
	\label{tab:mu}
\end{table}\\
We also exemplify the influence of the parameter $\mu$ on $\mathcal{RP}_{\mu}$'s performance. We choose the function $f_2$ as test instance because the $\mathcal{RP}_{\mu}$ consumes most FES on this function. We vary $\mu$ between $1\textsc{e}$-1 and  $1\textsc{e}$-10 and run the $\mathcal{RP}_{\mu}$ scheme 25 times to reach a relative accuracy of $1.91 \cdot 10^{-6}$. Mean number of ITS and FES are reported in Table \ref{tab:mu}. We see that the choice of $\mu$ has almost no influence on the number of ITS to reach the target accuracy, thus justifying the use of ITS as meaningful performance measure. The number of FES span the same order of magnitude ranging from $19824$ for $1\textsc{e}$-1 to $39070$ for $1\textsc{e}$-10. We see that the number of FES for the standard setting $\mu=1\textsc{e}$-5 is approximately in the middle of these extremes (29071 FES). This implies that the qualitative picture of the reported performance comparison is still valid but individual results for $\mathcal{RP}_{\mu}$ and $\mathcal{ARP}_{\mu}$ are improvable by optimally choosing $\mu$. An in-depth analysis of the optimal function-dependent choice of the $\mu$ parameter is subject to future studies.
 
As a final remark we highlight that the present numerical results for the Random Gradient methods are fully consistent with the ones presented in Nesterov's paper \cite{nesterov:randomgradientfree}.

\section{Discussion and Conclusion}
\label{sec:concl}

We have derived a convergence proof and convergence rates for Random Pursuit on convex functions. We have used a quadratic upper bound technique to bound the expected single-step progress of the algorithm.
Assuming exact line search, this results in global linear convergence for strongly convex functions and convergence of the order $1/k$ for general convex functions.

For line search oracles with relative error $\mu$ the same results have been obtained with convergence rates reduced by a factor of $\frac{1}{1-\mu}$. 
For inexact line search with absolute error $\mu$, convergence can be established only up to an additive error term depending on $\mu$, the properties of the function and the dimensionality.

The convergence rate of Random Pursuit exceeds the rate of the standard (first-order) Gradient Method by a factor of $n$. J\"{a}gersk\"{u}pper showed that no better performance can be expected for strongly convex functions \cite{jaeg:07}. He derived a lower bound for algorithms of the form~\eqref{eq:random} where at each iteration the step size along the random direction is chosen such as to minimize the distance to the minimum $x^*$. On sphere functions $f(x)=(x-x^*)^T(x-x^*)$ Random Pursuit coincides with the described scheme, thus achieving the lower bound. 

The numerical experiments showed that (i) standard Random Pursuit is effective on strongly convex functions with moderate condition number, and (ii) the accelerated scheme is comparable to Nesterov's fast gradient method and outperforms the $\mathcal{ES}$ algorithm. 
The experimental results also revealed that (i) $\mathcal{RP}_{\mu}$'s empirical convergence is (as predicted by theory) $n$ times slower than the one of the corresponding gradient scheme with line search ($\mathcal{GM}_{LS}$), and (ii) both continuous and discrete sampling can be employed in Random Pursuit. We confirmed the invariance of the $\mathcal{RP}_{\mu}$ algorithms and $\mathcal{ES}$ under monotone transformations of the objective functions on the quasiconvex funnel-shaped function $f_5$ where Random Gradient algorithms fail. We also highlighted the remarkable performance of the $\mathcal{ES}$ scheme given the fact that it does not need any function-specific input parameters. 

The present theoretical and experimental results hint at a number of potential enhancements for standard Random Pursuit in future work. First, $\mathcal{RP}_{\mu}$'s convergence rate depends on the function-specific parameter $L_1$ that bounds the curvature of the objective function. Any reduction of this dependency would imply faster convergence on a larger class of functions. The empirical results on the function pair $f_1/f_2$ (see Tables~S-1 and S-2 in~\cite{Stich:2011s}) also suggest that complicated accelerated schemes do not present any significant advantage on functions with small constant $L_1$. It is conceivable that Random Pursuit can incorporate a mechanism to learn second-order information about the function ``on the fly", thus improving the conditioning of the original optimization problem and potentially reducing it to the $L_1 \approx 1$ case. This may be possible using techniques from randomized Quasi-Newton approaches \cite{betro:78,leventhal:11,Stich:2012} or differential geometry \cite{Cheng:2005}. It is noteworthy that heuristic versions of such an adaptation mechanism have proved extremely useful in practice for adaptive step size algorithms \cite{Kjellstrom:1981,Igel:2006,Mueller:2010a}

Second, we have not considered Random Pursuit for constrained optimization problems of the form: 
\begin{align}
 \label{eq:problem_const}
 \min f(x) \quad \text{subject to} \quad x \in \mathcal{K} \, ,
\end{align}
where $\mathcal{K} \subset \reals^n$ is a convex set. The key challenge is how to treat iterates $x_{k+1}=x_k + \lsapprox(x_k,u)\cdot u$ generated by the line search oracle that are outside the domain $\mathcal{K}$. A classic idea is to apply a projection operator $\pi_{\mathcal{K}}$ and use the resulting $x'_{k+1} := \pi_{\mathcal{K}}(x_{k+1})$ as the next iterate. However, finding a projection onto a convex set (except for simple bodies such as hyper-parallelepipeds) can be as difficult as the original optimization problem. 
Moreover, it is an open question whether general convergence can be ensured, and what convergence rates can be achieved. Another possibility is to constrain the line search to the intersection of the line and the convex body $\mathcal{K}$. In this case, it is evident that one can only expect exponentially slow convergence rates for this method. Consider the linear function $f(x)=1^Tx$ and $\mathcal{K}=\reals^n_+$. Once an iterate $x_{k}$ lies at the boundary $\partial \mathcal{K}$ of the domain, say the first coordinate of $x_k$ is zero, then only directions $u$ with positive first coordinate may lead to an improvement. As soon as a constant fraction of the coordinates are zero, the probability of finding an improving direction is exponentially small. Karmanov~\cite{karmanov:74} proposed the following combination of projection and line search constraining: First, a random point $y$ at some fixed distance of the current iterate is drawn uniformly at random and then projected to the set $\mathcal{K}$. A constrained line search is now performed along the line through the current iterate $x_k$ and $\pi_{\mathcal{K}}(y)$. It remains open to study the convergence rate of this method. 
\begin{figure}[ht]
\centering
\def\svgwidth{0.85\columnwidth}
\includegraphics[width=\svgwidth]{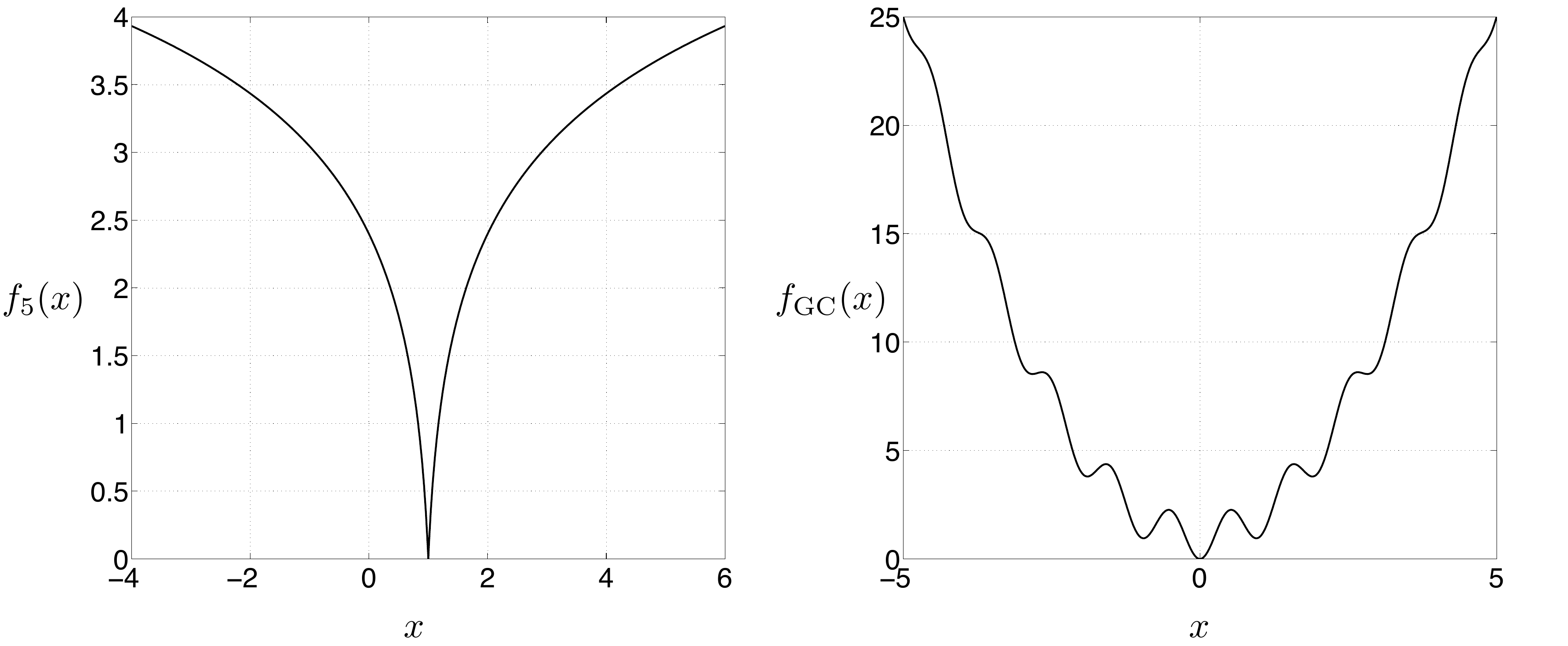}
\caption{Left panel: Graph of function $f_5$ in 1D. Right Panel:  Graph of a globally convex function $f_\text{GC}$.}
\label{fig:figure6}
\end{figure}

Finally, we envision convergence guarantees and provable convergence rates for Random Pursuit on more general function classes. The invariance of the line search oracle under strictly monotone transformations of the objective function already implied that Random Pursuit converges on certain strictly quasiconvex functions. 
It also seems in reach to derive convergence guarantees for Random Pursuit on the class of globally convex (or $\delta$-convex) functions \cite{Hu:1989} or on convex functions with bounded perturbations \cite{Phu:2010} (see right panel of Figure~\ref{fig:figure6} for the graph of such an instance). This may be achieved by appropriately adapting line search methods to these function classes. In summary, we believe that the theoretical and experimental results on Random Pursuit represent a promising first step toward the design of competitive derivative-free optimization methods that are easy to implement, possess theoretical convergence guarantees, and are useful in practice.      

\section*{Acknowledgments}
We sincerely thank Dr. Martin Jaggi for several helpful discussions. We would also like to thank the referees for their careful reading of the manuscript and their constructive comments that truly helped improve the quality of the manuscript.

\onecolumn
\bibliographystyle{siam}
\bibliography{bib}

\appendix
\section{Lemma}

\begin{lemma}
\label{lemma:ind}
Let $\{f_t\}_{t \in \naturals}$ be a sequence with $f_i \in \reals_+$. Suppose
\begin{align*}
f_{t+1} &\leq \left(1-\theta / t  \right) f_t + C \theta^2/t^2 + D \, , \quad \text{for } t \geq 1,
\end{align*}
for some constants $\theta > 1$, $C > 0$ and $D \geq 0$. Then it follows by induction that 
\begin{align*}
f_t \leq Q(\theta)/t + (t-1) D\, ,
\end{align*}
where
$Q(\theta) = \max \left\{\theta^2 C / (\theta - 1) , f_1 \right\}$.
\end{lemma}

A very similar result was stated without proof in \cite{nemirovski:robust} and also Hazan~\cite{hazan:2008} is using the same.

{\em Proof}.
For $t=1$ it holds that $f_1 \leq Q(\theta)$ by definition of $Q(\theta)$. Assume that the result holds for $t \geq 1$. 
If $Q(\theta) = \theta^2C/(\theta-1)$ then we deduce:
\begin{align*}
f_{t+1} %
&\leq \frac{\theta^2 C (t-\theta)}{(\theta-1)t^2}+ \frac{C \theta^2}{t^2} + \frac{(t-\theta) (t-1) D}{t} +  D\\%
 &= \frac{\theta^2 C (t-1)}{(\theta-1)t^2} +  \frac{D (t^2 - \theta (t-1))}{t}
 \leq \frac{\theta^2 C}{(\theta-1)(t+1)} + t D \, .
\end{align*}
If on the other hand $Q(\theta)=f_1$, then
\begin{align*}
f_1 \geq \frac{\theta^2 C}{(\theta-1)} \quad \Leftrightarrow \quad (\theta-1) f_1 \geq \theta^2 C \, ,
\end{align*}
and it follows
\begin{align*}
f_{t+1} &\leq \frac{(t-\theta)f_1}{t^2} + \frac{C \theta^2}{t^2} + \frac{(t-\theta) (t-1) D}{t} +  D  \\
&= \frac{(t-1)f_1}{t^2} + \frac{\theta^2C - (\theta-1)f_1}{t^2} + \frac{D (t^2 - \theta (t-1))}{t} 
\leq \frac{f_1}{t+1}  + t D \, . \qquad \endproof
\end{align*}

%
%

\section{Tables}
\subsection{Initial $\sigma_0$ of the $\mathcal{ES}$ algorithm for all test functions}
Table~\ref{tab:sigma} reports the empirically determined optimal initial step sizes $\sigma_0$ used as input to the $\mathcal{ES}$ algorithm.
\begin{table}[H]
	\centering
	\scalebox{0.58}{
	\begin{tabular}{ | r | r | r | r | r |}
  \hline
  dim  &  $f_1$ & $f_2$ & $f_3$ & $f_4$ \\
  \hline 
4    & 0.79158  &  1.3897  & 0.2054 & 0.20395\\
8    & 0.49167  &  0.78761 & 0.08922 & 0.088145\\
16   & 0.32692  &  0.49500 & 0.04134 &  0.041273\\
32   & 0.22292  &  0.32547 &  0.019911 & 0.019905\\
64   & 0.15542  &  0.22243 & 0.0097212 &  0.0097127\\
128  & 0.10925  &   0.15638 & 0.0048305 &  0.0048335\\
256  & 0.076658 & 0.10902 & 0.0024171 &  0.0024114\\
512  & 0.054339 & 0.076568 & 0.0012012 & 0.0012006\\
1024 & 0.038367 & 0.054173 & 0.00060284 & 0.00060223\\
  \hline  
\end{tabular}	

	}
  \caption{The initial values of the stepsize $\sigma$ for $(1+1)$-ES on the test functions for various dimensions.}
	\label{tab:sigma}
\end{table}
\subsection{Data for the ellipsoid test function for $\mathbf{n\le1024}$}
Tables~\ref{tab:ellipsoid2} and~\ref{tab:ellipsoid2_f} report the numerical data used to produce Figure~\ref{fig:figure1}.
\begin{table}[H]
	\centering
	\scalebox{0.58}{
	\begin{tabular}{ | r | r r r | r r r  | r r r |  r r r |  r r r |  r | r |}
  \hline
  \multicolumn{1}{|r|}{ } & \multicolumn{3}{c|}{$\mathcal{RP}$}  & \multicolumn{3}{c|}{$\mathcal{RG}$} 
						  & \multicolumn{3}{c|}{$\mathcal{FG}$}
                          & \multicolumn{3}{c|}{$\mathcal{ARP}$} & \multicolumn{3}{c|}{$\mathcal{ES}$}
						  & \multicolumn{1}{c|}{$\mathcal{GM}$} 
                          & \multicolumn{1}{c|}{$\mathcal{GM}_{LS}$} \\
  n  &  min & max & mean & min & max & mean & min & max & mean & min & max & mean & min & max & mean & - & -\\
  \hline 
4   &   236 &  472 &  364 &   28322 & 33608 & 31549   &    966 & 1575 & 1282 &   124 & 682 & 322 & 2491 & 4557 & 3784 &   3934 &  3\\
8   &   787 & 1241 & 1088 &   22461 & 24610 & 23666   &    981 & 1262 & 1155 &   174 & 285 & 226 & 3786 & 5799 & 4906 &   3934 &  3\\
16  &  1326 & 1763 & 1624 &   18981 & 20403 & 19805   &    975 & 1164 & 1076 &   218 & 256 & 232 & 4967 & 6034 & 5400 &   3934 &  3\\
32  &  1769 & 2026 & 1880 &   17381 & 18393 & 17858   &    968 & 1102 & 1048 &   221 & 256 & 237 & 5183 & 6145 & 5625 &   3934 &  3\\
64  &  1899 & 2096 & 2001 &   16601 & 17333 & 16868   &    990 & 1079 & 1038 &   233 & 250 & 242 & 5451 & 5954 & 5729 &   3934 &  3\\
128 &  1987 & 2145 & 2076 &   16183 & 16721 & 16376   &    978 & 1061 & 1030 &   237 & 252 & 245 & 5512 & 5964 & 5753 &   3934 &  3\\
256 &  2063 & 2173 & 2117 &   15960 & 16276 & 16115   &   1007 & 1053 & 1026 &   238 & 251 & 245 & 5603 & 6065 & 5805 &   3934 &  3\\
512 &  2081 & 2159 & 2119 &   15937 & 16103 & 16011   &   1015 & 1037 & 1026 &   239 & 249 & 244 & 5706 & 5909 & 5818 &   3934 &  3\\
1024&  2109 & 2152 & 2132 &   15846 & 16065 & 15955   &   1020 & 1035 & 1027 &   242 & 247 & 244 & 5759 & 5932 & 5833 &   3934 &  3\\
  \hline  
\end{tabular}	

	}
  \caption{Ellipsoid function $f_2$ to accuracy $1.91 \cdot 10^{-6}$, $S=50n$, $L=1000$. \#ITS$/n$, ($\mathcal{GM}$ and $\mathcal{GM}_{LS}$: \#ITS).}
	\label{tab:ellipsoid2}
\end{table}
\begin{table}[H]
	\centering
	\scalebox{0.58}{
	\begin{tabular}{ | r | r r r | r r r |  r r r |  r r r |  r r r |}
  \hline
  \multicolumn{1}{|r|}{ } & \multicolumn{3}{c|}{$\mathcal{RP}$}  & \multicolumn{3}{c|}{$\mathcal{RG}$} 
                          & \multicolumn{3}{c|}{$\mathcal{FG}$}
                          & \multicolumn{3}{c|}{$\mathcal{ARP}$} & \multicolumn{3}{c|}{$\mathcal{ES}$}	\\
  n  &  min & max & mean & min & max & mean & min & max & mean & min & max & mean & min & max & mean\\
  \hline 
4   &  3155 &  6236 & 4775  &   56643 & 67217 & 63097 &   1933 & 3150 & 2564 &   1408 & 6983 & 3631 &   2491 & 4557 & 3784 \\
8   & 11043 & 17124 & 15216 &   44923 & 49221 & 47331 &   1963 & 2525 & 2310 &   2113 & 3483 & 2774 &   3786 & 5799 & 4906 \\
16  & 19182 & 25225 & 23320 &   37962 & 40806 & 39610 &   1951 & 2329 & 2152 &   2730 & 3239 & 2930 &   4967 & 6034 & 5400 \\
32  & 25768 & 29258 & 27302 &   34762 & 36785 & 35715 &   1937 & 2203 & 2097 &   2870 & 3243 & 3043 &   5183 & 6145 & 5625 \\
64  & 27723 & 30272 & 29071 &   33202 & 34667 & 33736 &   1980 & 2159 & 2077 &   3035 & 3247 & 3159 &   5451 & 5954 & 5729 \\
128 & 28791 & 30757 & 29894 &   32365 & 33441 & 32753 &   1956 & 2121 & 2059 &   3139 & 3354 & 3259 &   5512 & 5964 & 5753 \\
256 & 29363 & 30691 & 30016 &   31920 & 32552 & 32230 &   2013 & 2106 & 2053 &   3210 & 3411 & 3322 &   5603 & 6065 & 5805 \\
512 & 29173 & 30167 & 29644 &   31874 & 32207 & 32022 &   2031 & 2075 & 2053 &   3311 & 3453 & 3386 &   5706 & 5909 & 5818 \\
1024& 29316 & 29914 & 29639 &   31692 & 32131 & 31910 &   2041 & 2069 & 2054 &   3455 & 3541 & 3494 &   5759 & 5932 & 5833 \\
  \hline  
\end{tabular}	

	}
  \caption{Ellipsoid function $f_2$ to accuracy $1.91 \cdot 10^{-6}$, $S=50n$, $L=1000$. \#FES$/n$.}
	\label{tab:ellipsoid2_f}
\end{table}
\cleardoublepage

\thispagestyle{plain}

\begin{center}
\textsc{Supporting Online Material for}\\[2em]
{ {\bfseries \MakeUppercase{{Optimization of Convex Functions with Random Pursuit}$^*$}}}
\end{center}
~\\
\begin{center}
\textsc{S.~U. Stich$^\dagger$, C.~L. M\"{u}ller$^\ddagger$, and B. G\"{a}rtner$^\S$}
\end{center}

\renewcommand{\thefigure}{S-\arabic{figure}}
\renewcommand{\thetable}{S-\arabic{table}}

\renewcommand{\thefootnote}{\fnsymbol{footnote}}
\footnotetext[1]{The project CG Learning acknowledges the financial support of the Future and Emerging Technologies (FET) programme within the Seventh Framework Programme for Research of the European Commission, under FET-Open grant number: 255827}
\footnotetext[2]{Institute of Theoretical Computer Science, ETH Z\"urich, and Swiss Institute of Bioinformatics, \texttt{sstich@inf.ethz.ch}}
\footnotetext[3]{Institute of Theoretical Computer Science, ETH Z\"urich, and Swiss Institute of Bioinformatics, \texttt{christian.mueller@inf.ethz.ch}}
\footnotetext[4]{Institute of Theoretical Computer Science, ETH Z\"urich, \texttt{gaertner@inf.ethz.ch}}
\renewcommand{\thefootnote}{\arabic{footnote}}

\newpage
\section{Exemplary Matlab Codes}

\subsection{Accelerated Random Pursuit}
~
\begin{figure}[H]
\begin{lstlisting}
function [fval, x, funeval] = rp_acc(fitfun, xstart, N, mu, m, L)
% Algorithm Accelerated Random Pursuit

% fitfun    name or function handle
% N         number of iterations
% mu        line search accuracy
% funeval   #function evaluations
% m, L      parameters for quadratic upper and lower bounds

%line search parameters
opts = optimset('Display', 'off', 'LargeScale', 'off', 'TolX', mu);
funeval = 0;
x = xstart;
n = length(xstart);

th = 1/(L*n^2);
ga = m;
v = x;

for i = 1:N
   p = [1/th, ga - m, -ga];
   be = max(roots(p));
   de = (be * ga) / (ga + be * m);
   y = (1 - de) * x + de * v;
   ga = (1 - be) * ga + be * m;
   la = be / ga * m;
   
   d=randn(size(xstart));  d=d/norm(d);
   [sigma, fval, ~, infos] = ...
             fminunc(@(sigma) feval(fitfun, y + sigma*d), 0, opts);
   funeval = funeval + infos.funcCount;
   x = y + sigma * d;
   v = (1 - la) * v + la * y + sigma/(be*n) * d;
end
\end{lstlisting}
\caption{Matlab code for algorithm $\mathcal{ARP}_\mu$}
\end{figure}
\newpage

\subsection{Random Pursuit}
~
\begin{figure}[H]
\begin{lstlisting}
function [fval, x, funeval] = rp(fitfun, xstart, N, mu)
% Algorithm Random Pursuit

% fitfun    name or function handle
% N         number of iterations
% mu        line search accuracy
% funeval   #function evaluations

%line search parameters
opts = optimset('Display', 'off', 'LargeScale', 'off', 'TolX', mu);
funeval = 0;
x = xstart;

for i = 1:N
    d=randn(size(xstart));  d=d/norm(d);
    [sigma, fval, ~, infos] = ...
             fminunc(@(sigma) feval(fitfun, x + sigma*d), 0, opts);
    funeval = funeval + infos.funcCount;
    x = x + sigma*d;
end
\end{lstlisting}%
\caption{Matlab code for algorithm $\mathcal{RP}_\mu$}
\label{fig:rp}
\end{figure}

\subsection{Random Gradient}
~
\begin{figure}[H]
\begin{lstlisting}
function [fval, x] = rg(fitfun, xstart, N, eps, L)
% Random Gradient [Nesterov 2011]

% fitfun    name or function handle
% N         number of iterations
% eps       finite difference parameter
% L         quadratic upper bound

x = xstart;
n = length(xstart);
h = 1/(4 * L *(n+4));

for i = 1:N
   d = randn(size(xstart)); 
   g = (feval(fitfun, x + eps*d) - feval(fitfun, x)) / eps;
   x =  x - h * g * d;
end
fval = feval(fitfun, x);
\end{lstlisting}
\caption{Matlab code for algorithm $\mathcal{RG}$}
\end{figure}
\newpage

\subsection{Accelerated Random Gradient}
~
\begin{figure}[H]
\begin{lstlisting}
function [fval, x] = rg_acc(fitfun, xstart, N, eps, m, L)
% Accelerated Random Gradient [Nesterov 2011]

% fitfun    name or function handle
% N         number of iterations
% eps       finite difference parameter
% m, L      parameters for quadratic upper and lower bounds

x = xstart;
n = length(xstart);
h = 1/(4 * L *(n+4));
th = 1/(16 * L * (n+1)^2);
ga = m;
v = x;

for i = 1:N
   p = [1/th, ga - m, -ga];
   be = max(roots(p));
   de = (be * ga) / (ga + be * m);
   y = (1 - de) * x + de * v;
   ga = (1 - be) * ga + be * m;
   la = be / ga * m;
    
   d=randn(size(xstart));
   g = (feval(fitfun, x + eps*d) - feval(fitfun, x)) / eps;
   x = y - h * g * d;
   v = (1 - la) * v + la * y - th / be * g * d;
end
fval = feval(fitfun, x);
\end{lstlisting}
\caption{Matlab code for algorithm $\mathcal{FG}$}
\end{figure}
\newpage

\subsection{(1+1)-ES}
~
\begin{figure}[H]
\begin{lstlisting}
function [fval, x] = es(fitfun, xstart, N, sigma)
% (1+1)-ES

% fitfun    name or function handle
% N         number of iterations
% sigma     initial stepsize
% funeval   #function evaluations

fval = feval(fitfun, xstart);
x = xstart;
ss=exp(1/3);
ff=exp(1/3 * (-0.27)/(1-0.27));

for i = 1:N
   d = randn(size(xstart)); 
   f = feval(fitfun, x + sigma*d);
 
   if f <= fval
       x = x + sigma * d; fval = f; sigma = sigma * ss;
   else
       sigma = sigma * ff;
   end
end
\end{lstlisting}
\caption{Matlab code for algorithm $\mathcal{ES}$}
\end{figure}

\section{Tables}
%

%
%

\subsection{Number of iterations for increasing accuracy for $\mathbf{n=64}$}

Tables~\ref{tab:sphere} - \ref{tab:funnel} summarize the number of iterations needed to achieve a corresponding relative accuracy (acc) for fixed dimension $n=64$.

\begin{table}[H]
	\centering
	\scalebox{0.57}{
	\begin{tabular}{ | r | r r r | r r r |  r r r |  r r r |  r r r | r | r |}
  \hline
  \multicolumn{1}{|r|}{ } & \multicolumn{3}{c|}{$\mathcal{RP}$}  & \multicolumn{3}{c|}{$\mathcal{RG}$} 
						  & \multicolumn{3}{c|}{$\mathcal{FG}$}
                          & \multicolumn{3}{c|}{$\mathcal{ARP}$} & \multicolumn{3}{c|}{$\mathcal{ES}$}
						  & \multicolumn{1}{c|}{$\mathcal{GM}$}  
						  & \multicolumn{1}{c|}{$\mathcal{GM}_{LS}$} \\
  acc.  &  min & max & mean & min & max & mean  & min & max & mean & min & max & mean & min & max & mean & - & -\\
  \hline 
$6.25 \cdot 10^{-2}$ &  2  &  3  &  3  &  6  &  8  &  7     &    6 &  8 &  7 &   2  &  3  &  3  & 6  & 9  & 8  & 1     & 1   \\
$3.12 \cdot 10^{-2}$ &  3  &  4  &  3  &  8  &  10 &  8     &    8 & 10 &  8 &   3  &  4  &  3  & 8  & 12 & 10 & 1     & 1   \\ 
$1.56 \cdot 10^{-2}$ &  4  &  5  &  4  &  9  &  12 &  10    &    9 & 12 & 10 &   4  &  5  &  4  & 10 & 14 & 12 & 1     & 1   \\
$7.81 \cdot 10^{-3}$ &  4  &  6  &  5  &  11 &  13 &  12    &   11 & 13 & 12 &   4  &  5  &  5  & 12 & 16 & 14 & 1     & 1   \\
$3.91 \cdot 10^{-3}$ &  5  &  6  &  5  &  13 &  15 &  14    &   13 & 15 & 14 &   5  &  6  &  5  & 14 & 18 & 16 & 1     & 1   \\ 
$1.95 \cdot 10^{-3}$ &  5  &  7  &  6  &  14 &  17 &  15    &   14 & 17 & 15 &   5  &  7  &  6  & 15 & 20 & 18 & 1     & 1   \\
$9.77 \cdot 10^{-4}$ &  6  &  8  &  7  &  16 &  18 &  17    &   16 & 18 & 17 &   6  &  8  &  7  & 17 & 22 & 20 & 1     & 1   \\
$4.88 \cdot 10^{-4}$ &  7  &  9  &  7  &  18 &  20 &  19    &   18 & 20 & 19 &   7  &  9  &  7  & 18 & 24 & 21 & 1     & 1   \\
$2.44 \cdot 10^{-4}$ &  7  &  9  &  8  &  19 &  22 &  20    &   19 & 22 & 20 &   7  &  9  &  8  & 20 & 26 & 23 & 1     & 1   \\
$1.22 \cdot 10^{-4}$ &  8  &  10 &  9  &  21 &  24 &  22    &   21 & 24 & 22 &   8  &  10 &  9  & 22 & 28 & 25 & 1     & 1   \\
$6.10 \cdot 10^{-5}$ &  9  &  10 &  10 &  22 &  26 &  24    &   22 & 26 & 24 &   9  &  11 &  9  & 23 & 31 & 27 & 1     & 1   \\
$3.05 \cdot 10^{-5}$ &  9  &  11 &  10 &  24 &  28 &  25    &   24 & 28 & 25 &   10 &  11 &  10 & 26 & 32 & 29 & 1     & 1   \\
$1.53 \cdot 10^{-5}$ &  10 &  12 &  11 &  25 &  29 &  27    &   25 & 29 & 27 &   10 &  12 &  11 & 28 & 35 & 31 & 1     & 1   \\
$7.63 \cdot 10^{-6}$ &  11 &  13 &  12 &  27 &  31 &  29    &   27 & 31 & 29 &   11 &  13 &  12 & 29 & 36 & 33 & 1     & 1   \\
$3.81 \cdot 10^{-6}$ &  11 &  13 &  12 &  28 &  32 &  30    &   28 & 32 & 30 &   12 &  13 &  12 & 31 & 38 & 35 & 1     & 1   \\
$1.91 \cdot 10^{-6}$ &  12 &  14 &  13 &  30 &  34 &  32    &   30 & 34 & 32 &   12 &  14 &  13 & 33 & 41 & 37 & 1     & 1   \\
  \hline  
\end{tabular}	

	}
	\caption{Sphere function $f_1$, $m=1$, $L=1$, $S=32$, $n=64$. \#ITS$/n$, ($\mathcal{GM}$, $\mathcal{GM}_{LS}$: \#ITS).}
	\label{tab:sphere}
\end{table}

\begin{table}[H]
	\centering
	\scalebox{0.57}{
	\begin{tabular}{ | r | r r r | r r r | r r r |  r r r |  r r r |  r | r |}
  \hline
  \multicolumn{1}{|r|}{ } & \multicolumn{3}{c|}{$\mathcal{RP}$}  & \multicolumn{3}{c|}{$\mathcal{RG}$} 
						  & \multicolumn{3}{c|}{$\mathcal{FG}$}
                          & \multicolumn{3}{c|}{$\mathcal{ARP}$} & \multicolumn{3}{c|}{$\mathcal{ES}$}	
                          & \multicolumn{1}{c|}{$\mathcal{GM}$} 
                          & \multicolumn{1}{c|}{$\mathcal{GM}_{LS}$} 						  \\
  acc.  &  min & max & mean & min & max & mean   & min & max & mean & min & max & mean & min & max & mean & - & -\\
  \hline 
$6.25 \cdot 10^{-2}$ & 2   & 3   & 2    & 9     &  11    &  10      &  5  & 18  & 16   &  64  &  79  &  71  &  5    & 9    & 6    &   1    &   1 \\
$3.12 \cdot 10^{-2}$ & 2   & 3   & 3    & 11    &  13    &  12      &  17 & 21  & 18   &  70  &  86  &  77  &  6    & 10   & 8    &   1    &   1 \\
$1.56 \cdot 10^{-2}$ & 3   & 4   & 3    & 13    &  15    &  14      &  31 & 359 & 261  &  77  &  93  &  84  &  7    & 19   & 10   &   1    &   1 \\
$7.81 \cdot 10^{-3}$ & 4   & 132 & 53   & 15    &  20    &  17      &  340& 423 & 380  &  84  &  101 &  91  &  18   & 465  & 268  &   1    &   1 \\
$3.91 \cdot 10^{-3}$ & 104 & 298 & 210  & 381   &  1103  &  677     &  404& 484 & 444  &  92  &  125 &  107 &  481  & 928  & 723  &   124  &   2 \\
$1.95 \cdot 10^{-3}$ & 273 & 460 & 373  & 1863  &  2578  &  2150    &  464& 544 & 504  &  101 &  140 &  129 &  936  & 1410 & 1177 &   470  &   2 \\
$9.77 \cdot 10^{-4}$ & 433 & 624 & 536  & 3340  &  4062  &  3624    &  521& 601 & 562  &  129 &  153 &  143 &  1376 & 1869 & 1631 &   817  &   2 \\
$4.88 \cdot 10^{-4}$ & 598 & 787 & 700  & 4815  &  5536  &  5094    &  576& 657 & 618  &  142 &  164 &  153 &  1826 & 2325 & 2085 &   1163 &   2 \\
$2.44 \cdot 10^{-4}$ & 761 & 954 & 862  & 6280  &  7018  &  6564    &  630& 712 & 673  &  153 &  174 &  162 &  2264 & 2774 & 2538 &   1509 &   2 \\
$1.22 \cdot 10^{-4}$ & 921 & 1118& 1024 & 7754  &  8488  &  8036    &  683& 767 & 727  &  161 &  183 &  170 &  2732 & 3239 & 2994 &   1856 &   2 \\
$6.10 \cdot 10^{-5}$ & 1080& 1284& 1187 & 9232  &  9961  &  9508    &  736& 820 & 781  &  168 &  191 &  178 &  3172 & 3690 & 3447 &   2202 &   2 \\
$3.05 \cdot 10^{-5}$ & 1243& 1446& 1350 & 10712 &  11439 &  10980   &  787& 873 & 833  &  176 &  199 &  187 &  3635 & 4138 & 3905 &   2549 &   3 \\
$1.53 \cdot 10^{-5}$ & 1406& 1607& 1512 & 12179 &  12906 &  12453   &  839& 925 & 885  &  189 &  214 &  201 &  4083 & 4593 & 4361 &   2895 &   3 \\
$7.63 \cdot 10^{-6}$ & 1570& 1766& 1675 & 13654 &  14388 &  13923   &  890& 977 & 936  &  203 &  227 &  217 &  4537 & 5042 & 4819 &   3241 &   3 \\
$3.81 \cdot 10^{-6}$ & 1732& 1928& 1837 & 15130 &  15854 &  15395   &  940& 1029& 988  &  219 &  238 &  230 &  4989 & 5492 & 5273 &   3588 &   3 \\
$1.91 \cdot 10^{-6}$ & 1899& 2096& 2001 & 16601 &  17333 &  16868   &  990& 1079& 1038 &  233 &  250 &  242 &  5451 & 5954 & 5729 &   3934 &   3 \\
  \hline  
\end{tabular}	

	}
  \caption{Ellipsoid function $f_2$, $m=1$, $L=1000$, $S=3200$, $n=64$. \#ITS$/n$, ($\mathcal{GM}$, $\mathcal{GM}_{LS}$: \#ITS).}
	\label{tab:ellipsoid}
\end{table}

\begin{table}[H]
	\centering
	\scalebox{0.57}{
	\begin{tabular}{ | r | r r r | r r r |  r r r |  r r r |  r r r | r | r |}
  \hline
  \multicolumn{1}{|r|}{ } & \multicolumn{3}{c|}{$\mathcal{RP}$}  & \multicolumn{3}{c|}{$\mathcal{RG}$} 
						  & \multicolumn{3}{c|}{$\mathcal{FG}$}
                          & \multicolumn{3}{c|}{$\mathcal{ARP}$} & \multicolumn{3}{c|}{$\mathcal{ES}$}
						  & \multicolumn{1}{c|}{$\mathcal{GM}$} 
                          & \multicolumn{1}{c|}{$\mathcal{GM}_{LS}$} \\
  acc.  &  min & max & mean & min & max & mean   & min & max & mean & min & max & mean & min & max & mean & - & -\\
  \hline 
$6.25 \cdot 10^{-2}$ &  0   & 0   & 0     &  0     &  0     &  0       &    0 &   0 &   0 &    0 &   0 &   0 &    0 & 0    & 0    &  1    &  1   \\
$3.12 \cdot 10^{-2}$ &  0   & 0   & 0     &  0     &  0     &  0       &    0 &   0 &   0 &    0 &   0 &   0 &    0 & 0    & 0    &  1    &  1   \\
$1.56 \cdot 10^{-2}$ &  0   & 0   & 0     &  0     &  0     &  0       &    0 &   0 &   0 &    0 &   0 &   0 &    0 & 0    & 0    &  1    &  1   \\
$7.81 \cdot 10^{-3}$ &  1   & 1   & 1     &  3     &  5     &  4       &    2 &   4 &   3 &    0 &   1 &   1 &    2 & 4    & 3    &  1    &  1   \\
$3.91 \cdot 10^{-3}$ &  3   & 4   & 3     &  18    &  23    &  21      &    8 &  13 &  10 &    3 &  19 &   8 &    6 & 10   & 9    &  5    &  3   \\
$1.95 \cdot 10^{-3}$ &  8   & 12  & 10    &  74    &  84    &  79      &   29 &  40 &  34 &   12 &  66 &  25 &   22 & 31   & 26   &  19   &  10  \\
$9.77 \cdot 10^{-4}$ &  26  & 37  & 31    &  256   &  283   &  269     &   55 &  84 &  70 &   22 &  75 &  41 &   73 & 102  & 86   &  64   &  32  \\
$4.88 \cdot 10^{-4}$ &  79  & 109 & 92    &  790   &  837   &  811     &  104 & 152 & 130 &   33 & 132 &  61 &  233 & 292  & 257  &  191  &  96  \\
$2.44 \cdot 10^{-4}$ &  200 & 264 & 228   &  1993  &  2065  &  2022    &  164 & 258 & 201 &   35 & 173 &  87 &  577 & 700  & 633  &  477  &  239 \\
$1.22 \cdot 10^{-4}$ &  405 & 501 & 453   &  3955  &  4053  &  4004    &  224 & 328 & 279 &   58 & 199 & 127 & 1142 & 1344 & 1249 &  945  &  473 \\
$6.10 \cdot 10^{-5}$ &  665 & 778 & 723   &  6349  &  6465  &  6412    &  293 & 382 & 348 &   74 & 255 & 151 & 1867 & 2101 & 1998 &  1512 &  756 \\
$3.05 \cdot 10^{-5}$ &  948 & 1060& 1005  &  8849  &  8964  &  8917    &  369 & 427 & 402 &   88 & 391 & 213 & 2641 & 2891 & 2780 &  2101 &  1051\\
$1.53 \cdot 10^{-5}$ &  1229& 1345& 1288  &  11375 &  11482 &  11435   &  397 & 613 & 442 &   96 & 449 & 265 & 3406 & 3692 & 3563 &  2694 &  1347\\
$7.63 \cdot 10^{-6}$ &  1509& 1619& 1570  &  13886 &  14016 &  13958   &  450 & 894 & 677 &  129 & 533 & 342 & 4223 & 4461 & 4348 &  3288 &  1644\\
$3.81 \cdot 10^{-6}$ &  1792& 1902& 1853  &  16401 &  16545 &  16482   &  463 & 935 & 887 &  188 & 632 & 400 & 5008 & 5254 & 5133 &  3881 &  1940\\
$1.91 \cdot 10^{-6}$ &  2068& 2191& 2136  &  18922 &  19075 &  19004   &  892 & 970 & 942 &  192 & 678 & 473 & 5766 & 6050 & 5916 &  4474 &  2237\\
  \hline  
\end{tabular}	

	}
  \caption{Nesterov smooth $f_3$, $L=1000$, $S=10833$, $n=64$. \#ITS$/n$, ($\mathcal{GM}$, $\mathcal{GM}_{LS}$: \#ITS).}
	\label{tab:nesterovsmooth}
\end{table}

\begin{table}[H]
	\centering
	\scalebox{0.57}{
	\begin{tabular}{ | r | r r r | r r r | r r r |  r r r |  r r r | r | r | }
  \hline
  \multicolumn{1}{|r|}{ } & \multicolumn{3}{c|}{$\mathcal{RP}$}  & \multicolumn{3}{c|}{$\mathcal{RG}$} 					  
						  & \multicolumn{3}{c|}{$\mathcal{FG}$}
                          & \multicolumn{3}{c|}{$\mathcal{ARP}$} & \multicolumn{3}{c|}{$\mathcal{ES}$}
						  & \multicolumn{1}{c|}{$\mathcal{GM}$}  
                          & \multicolumn{1}{c|}{$\mathcal{GM}_{LS}$}  	\\
  acc.  &  min & max & mean & min & max & mean   & min & max & mean & min & max & mean & min & max & mean & - & -\\
  \hline 
$6.25 \cdot 10^{-2}$ &  1   &  2    &  1   &  7    &  9    &  8      &    4 &   5 &   5 &  1   &  6   &  2   &  3   & 5   & 4     &  2    &  1   \\
$3.12 \cdot 10^{-2}$ &  3   &  5    &  4   &  25   &  31   &  28     &   10 &  16 &  13 &  5   &  19  &  11  &  8   & 13  & 11    &  7    &  4   \\
$1.56 \cdot 10^{-2}$ &  8   &  12   &  10  &  79   &  90   &  82     &   27 &  37 &  33 &  9   &  45  &  22  &  19  & 35  & 27    &  19   &  10  \\
$7.81 \cdot 10^{-3}$ &  19  &  31   &  24  &  199  &  219  &  204    &   46 &  70 &  59 &  21  &  53  &  32  &  52  & 76  & 65    &  48   &  24  \\
$3.91 \cdot 10^{-3}$ &  43  &  62   &  50  &  415  &  458  &  432    &   74 & 107 &  88 &  25  &  58  &  41  &  109 & 152 & 135   &  102  &  51  \\
$1.95 \cdot 10^{-3}$ &  82  &  106  &  90  &  772  &  824  &  791    &  104 & 140 & 118 &  39  &  88  &  53  &  214 & 272 & 247   &  187  &  94  \\
$9.77 \cdot 10^{-4}$ &  131 &  166  &  146 &  1248 &  1341 &  1284   &  138 & 177 & 157 &  49  &  96  &  65  &  352 & 447 & 399   &  303  &  152 \\
$4.88 \cdot 10^{-4}$ &  195 &  236  &  214 &  1841 &  1965 &  1900   &  169 & 216 & 191 &  52  &  109 &  73  &  533 & 638 & 587   &  449  &  225 \\
$2.44 \cdot 10^{-4}$ &  267 &  317  &  295 &  2540 &  2694 &  2621   &  195 & 259 & 228 &  64  &  114 &  82  &  740 & 875 & 810   &  619  &  310 \\
$1.22 \cdot 10^{-4}$ &  352 &  409  &  384 &  3326 &  3513 &  3420   &  244 & 293 & 266 &  68  &  127 &  95  &  976 & 1132& 1059  &  807  &  404 \\
$6.10 \cdot 10^{-5}$ &  441 &  506  &  479 &  4174 &  4387 &  4277   &  282 & 335 & 306 &  89  &  132 &  107 &  1233& 1403& 1325  &  1009 &  505 \\
$3.05 \cdot 10^{-5}$ &  539 &  605  &  579 &  5057 &  5288 &  5168   &  321 & 376 & 342 &  96  &  160 &  119 &  1508& 1681& 1603  &  1219 &  610 \\
$1.53 \cdot 10^{-5}$ &  642 &  715  &  682 &  5969 &  6206 &  6079   &  351 & 402 & 376 &  108 &  168 &  130 &  1788& 1974& 1885  &  1434 &  717 \\
$7.63 \cdot 10^{-6}$ &  740 &  816  &  786 &  6893 &  7138 &  7001   &  383 & 430 & 409 &  117 &  177 &  138 &  2071& 2269& 2173  &  1650 &  826 \\
$3.81 \cdot 10^{-6}$ &  845 &  920  &  890 &  7816 &  8064 &  7926   &  423 & 453 & 435 &  127 &  184 &  148 &  2357& 2568& 2462  &  1868 &  935 \\
$1.91 \cdot 10^{-6}$ &  954 &  1023 &  995 &  8727 &  8995 &  8854   &  441 & 534 & 458 &  137 &  188 &  159 &  2651& 2854& 2751  &  2086 &  1044\\
  \hline  
\end{tabular}	

	}
  \caption{Nesterov strongly convex function $f_4$, $m=1$, $L=1000$, $S=1000$, $n=64$. \#ITS$/n$, ($\mathcal{GM}$, $\mathcal{GM}_{LS}$: \#ITS),}
	\label{tab:nesterovstrong}
\end{table}

\begin{table}[H]
	\centering
	\scalebox{0.57}{
	\begin{tabular}{ | r | r r r | r r r |  r r r |  r r r |  r r r | r | r |}
  \hline
  \multicolumn{1}{|r|}{ } & \multicolumn{3}{c|}{$\mathcal{RP}$}  & \multicolumn{3}{c|}{$\mathcal{RG}$} 
						  & \multicolumn{3}{c|}{$\mathcal{FG}$}
                          & \multicolumn{3}{c|}{$\mathcal{ARP}$} & \multicolumn{3}{c|}{$\mathcal{ES}$}	
                          & \multicolumn{1}{c|}{$\mathcal{GM}$} 
                          & \multicolumn{1}{c|}{$\mathcal{GM}_{LS}$} 						  \\
  acc.  &  min & max & mean & min & max & mean & min & max & mean & min & max & mean & min & max & mean & - & - \\
  \hline 
$6.25 \cdot 10^{-2}$ &   4 &  6 &  5 &  - &  - &  -    &  - & - & - &    4 &  6 &  5 & 13 & 17 & 14 & -    &  1 \\
$3.12 \cdot 10^{-2}$ &   7 &  9 &  7 &  - &  - &  -    &  - & - & - &    7 &  8 &  7 & 19 & 25 & 22 & -    &  1 \\ 
$1.56 \cdot 10^{-2}$ &   8 & 11 &  9 &  - &  - &  -    &  - & - & - &    8 & 10 &  9 & 24 & 32 & 27 & -    &  1 \\
$7.81 \cdot 10^{-3}$ &  10 & 13 & 11 &  - &  - &  -    &  - & - & - &   10 & 12 & 11 & 29 & 37 & 32 & -    &  1 \\
$3.91 \cdot 10^{-3}$ &  11 & 15 & 12 &  - &  - &  -    &  - & - & - &   12 & 14 & 12 & 32 & 42 & 36 & -    &  1 \\ 
$1.95 \cdot 10^{-3}$ &  13 & 16 & 14 &  - &  - &  -    &  - & - & - &   13 & 15 & 14 & 35 & 46 & 40 & -    &  1 \\
$9.77 \cdot 10^{-4}$ &  14 & 17 & 15 &  - &  - &  -    &  - & - & - &   14 & 17 & 15 & 39 & 50 & 43 & -    &  1 \\
$4.88 \cdot 10^{-4}$ &  16 & 19 & 17 &  - &  - &  -    &  - & - & - &   15 & 18 & 17 & 43 & 54 & 47 & -    &  1 \\
$2.44 \cdot 10^{-4}$ &  17 & 20 & 18 &  - &  - &  -    &  - & - & - &   17 & 20 & 18 & 47 & 58 & 51 & -    &  1 \\
$1.22 \cdot 10^{-4}$ &  18 & 22 & 19 &  - &  - &  -    &  - & - & - &   18 & 22 & 19 & 51 & 62 & 55 & -    &  1 \\
$6.10 \cdot 10^{-5}$ &  19 & 23 & 21 &  - &  - &  -    &  - & - & - &   19 & 23 & 21 & 55 & 66 & 59 & -    &  1 \\
$3.05 \cdot 10^{-5}$ &  21 & 24 & 22 &  - &  - &  -    &  - & - & - &   21 & 25 & 22 & 59 & 70 & 63 & -    &  1 \\
$1.53 \cdot 10^{-5}$ &  22 & 25 & 23 &  - &  - &  -    &  - & - & - &   22 & 26 & 23 & 62 & 74 & 67 & -    &  1 \\
$7.63 \cdot 10^{-6}$ &  23 & 27 & 25 &  - &  - &  -    &  - & - & - &   23 & 28 & 25 & 67 & 77 & 70 & -    &  1 \\
$3.81 \cdot 10^{-6}$ &  25 & 28 & 26 &  - &  - &  -    &  - & - & - &   24 & 29 & 26 & 71 & 81 & 74 & -    &  1 \\
$1.91 \cdot 10^{-6}$ &  26 & 30 & 28 &  - &  - &  -    &  - & - & - &   26 & 30 & 28 & 73 & 85 & 78 & -    &  1 \\
  \hline  
\end{tabular}	

	}
  \caption{Funnel function $f_5$, $S=32$, $n=64$. \#ITS$/n$, ($\mathcal{GM}$, $\mathcal{GM}_{LS}$: \#ITS).}
	\label{tab:funnel}
\end{table}
%
%
\subsection{Number of function evaluations for increasing accuracy for fixed dimension $n=64$}

Tables~\ref{tab:sphere_f} - \ref{tab:funnel_f} summarize the number of function evaluations needed to achieve a corresponding relative accuracy (acc) for fixed dimension $n=64$.

\begin{table}[H]
	\centering
	\scalebox{0.6}{
	\begin{tabular}{ | r | r r r | r r r |  r r r |  r r r |  r r r |}
  \hline
  \multicolumn{1}{|r|}{ } & \multicolumn{3}{c|}{$\mathcal{RP}$}  & \multicolumn{3}{c|}{$\mathcal{RG}$} 
                          & \multicolumn{3}{c|}{$\mathcal{FG}$}
                          & \multicolumn{3}{c|}{$\mathcal{ARP}$} & \multicolumn{3}{c|}{$\mathcal{ES}$}	\\
  acc.  &  min & max & mean & min & max & mean & min & max & mean & min & max & mean & min & max & mean\\
  \hline 
$6.25 \cdot 10^{-2}$ &  10 & 14 & 12 &   11 & 15 & 14 &   12 & 15 & 13 &    9 & 13 & 11 &    6 &  9 &  8 \\
$3.12 \cdot 10^{-2}$ &  12 & 17 & 14 &   15 & 19 & 17 &   15 & 18 & 17 &   12 & 16 & 14 &    8 & 12 & 10 \\ 
$1.56 \cdot 10^{-2}$ &  15 & 20 & 17 &   18 & 23 & 20 &   18 & 21 & 20 &   15 & 19 & 17 &   10 & 14 & 12 \\
$7.81 \cdot 10^{-3}$ &  17 & 23 & 20 &   22 & 27 & 24 &   21 & 25 & 23 &   17 & 22 & 20 &   12 & 16 & 14 \\
$3.91 \cdot 10^{-3}$ &  20 & 27 & 22 &   25 & 30 & 27 &   24 & 28 & 26 &   20 & 25 & 22 &   14 & 18 & 16 \\ 
$1.95 \cdot 10^{-3}$ &  22 & 29 & 25 &   28 & 34 & 30 &   27 & 32 & 29 &   22 & 29 & 25 &   15 & 20 & 18 \\
$9.77 \cdot 10^{-4}$ &  25 & 33 & 28 &   32 & 37 & 34 &   29 & 35 & 33 &   24 & 32 & 27 &   17 & 22 & 20 \\
$4.88 \cdot 10^{-4}$ &  27 & 35 & 31 &   35 & 40 & 37 &   32 & 38 & 36 &   27 & 35 & 30 &   18 & 24 & 21 \\
$2.44 \cdot 10^{-4}$ &  30 & 38 & 33 &   38 & 44 & 41 &   35 & 42 & 39 &   29 & 38 & 33 &   20 & 26 & 23 \\
$1.22 \cdot 10^{-4}$ &  32 & 40 & 36 &   41 & 47 & 44 &   38 & 45 & 42 &   32 & 40 & 36 &   22 & 28 & 25 \\
$6.10 \cdot 10^{-5}$ &  35 & 42 & 39 &   44 & 52 & 47 &   41 & 48 & 45 &   36 & 43 & 38 &   23 & 31 & 27 \\
$3.05 \cdot 10^{-5}$ &  38 & 44 & 42 &   47 & 56 & 51 &   44 & 51 & 49 &   39 & 46 & 41 &   26 & 32 & 29 \\
$1.53 \cdot 10^{-5}$ &  41 & 48 & 44 &   51 & 58 & 54 &   47 & 55 & 52 &   41 & 49 & 44 &   28 & 35 & 31 \\
$7.63 \cdot 10^{-6}$ &  43 & 52 & 47 &   53 & 61 & 57 &   51 & 58 & 55 &   44 & 51 & 47 &   29 & 36 & 33 \\
$3.81 \cdot 10^{-6}$ &  45 & 54 & 50 &   57 & 65 & 60 &   54 & 63 & 58 &   47 & 53 & 49 &   31 & 38 & 35 \\
$1.91 \cdot 10^{-6}$ &  47 & 57 & 52 &   60 & 68 & 64 &   57 & 66 & 61 &   50 & 56 & 52 &   33 & 41 & 37 \\
  \hline  
\end{tabular}	

	}
	\caption{Sphere function $f_1$, $m=1$, $L=1$, $S=32$, $n=64$. \#FES$/n$.}
	\label{tab:sphere_f}
\end{table}

\begin{table}[H]
	\centering
	\scalebox{0.6}{
	\begin{tabular}{ | r | r r r | r r r |  r r r |  r r r |  r r r |}
  \hline
  \multicolumn{1}{|r|}{ } & \multicolumn{3}{c|}{$\mathcal{RP}$}  & \multicolumn{3}{c|}{$\mathcal{RG}$} 
                          & \multicolumn{3}{c|}{$\mathcal{FG}$}
                          & \multicolumn{3}{c|}{$\mathcal{ARP}$} & \multicolumn{3}{c|}{$\mathcal{ES}$}	\\
  acc.  &  min & max & mean & min & max & mean & min & max & mean & min & max & mean & min & max & mean\\
  \hline 
$6.25 \cdot 10^{-2}$ &    12 &    21 &    16 &      18 &    22 &    20 &     10 &   36 &   33 &    637 &  809 &  733 &      5 &    9 &    6 \\
$3.12 \cdot 10^{-2}$ &    15 &    25 &    19 &      21 &    25 &    23 &     34 &   41 &   37 &    705 &  883 &  802 &      6 &   10 &    8 \\
$1.56 \cdot 10^{-2}$ &    20 &    32 &    25 &      25 &    30 &    28 &     62 &  717 &  522 &    777 &  966 &  872 &      7 &   19 &   10 \\
$7.81 \cdot 10^{-3}$ &    29 &  1671 &   662 &      30 &    39 &    34 &    680 &  847 &  760 &    863 & 1092 &  963 &     18 &  465 &  268 \\
$3.91 \cdot 10^{-3}$ &  1416 &  3925 &  2803 &     761 &  2206 &  1354 &    808 &  968 &  887 &    978 & 1418 & 1185 &    481 &  928 &  723 \\
$1.95 \cdot 10^{-3}$ &  3809 &  6240 &  5125 &    3725 &  5155 &  4300 &    928 & 1087 & 1008 &   1054 & 1613 & 1489 &    936 & 1410 & 1177 \\
$9.77 \cdot 10^{-4}$ &  6096 &  8577 &  7458 &    6681 &  8124 &  7248 &   1042 & 1203 & 1124 &   1526 & 1790 & 1688 &   1376 & 1869 & 1631 \\
$4.88 \cdot 10^{-4}$ &  8495 & 10961 &  9848 &    9629 & 11072 & 10189 &   1152 & 1315 & 1236 &   1702 & 1955 & 1835 &   1826 & 2325 & 2085 \\
$2.44 \cdot 10^{-4}$ & 10936 & 13455 & 12278 &   12560 & 14036 & 13129 &   1259 & 1425 & 1347 &   1844 & 2108 & 1966 &   2264 & 2774 & 2538 \\
$1.22 \cdot 10^{-4}$ & 13340 & 15896 & 14705 &   15507 & 16977 & 16072 &   1365 & 1533 & 1455 &   1960 & 2242 & 2088 &   2732 & 3239 & 2994 \\
$6.10 \cdot 10^{-5}$ & 15726 & 18390 & 17144 &   18463 & 19921 & 19015 &   1471 & 1641 & 1562 &   2068 & 2362 & 2211 &   3172 & 3690 & 3447 \\
$3.05 \cdot 10^{-5}$ & 18147 & 20800 & 19566 &   21423 & 22877 & 21960 &   1575 & 1747 & 1666 &   2188 & 2508 & 2353 &   3635 & 4138 & 3905 \\
$1.53 \cdot 10^{-5}$ & 20573 & 23183 & 21970 &   24357 & 25812 & 24906 &   1678 & 1851 & 1770 &   2389 & 2743 & 2552 &   4083 & 4593 & 4361 \\
$7.63 \cdot 10^{-6}$ & 22981 & 25521 & 24375 &   27308 & 28776 & 27846 &   1780 & 1954 & 1873 &   2627 & 2918 & 2789 &   4537 & 5042 & 4819 \\
$3.81 \cdot 10^{-6}$ & 25337 & 27890 & 26733 &   30259 & 31707 & 30790 &   1880 & 2057 & 1975 &   2857 & 3083 & 2993 &   4989 & 5492 & 5273 \\
$1.91 \cdot 10^{-6}$ & 27723 & 30272 & 29071 &   33202 & 34667 & 33736 &   1980 & 2159 & 2077 &   3035 & 3247 & 3159 &   5451 & 5954 & 5729 \\
  \hline  
\end{tabular}	

	}
  \caption{Ellipsoid function $f_2$, $m=1$, $L=1000$, $S=3200$, $n=64$. \#FES/$n$.}
	\label{tab:ellipsoid_f}
\end{table}

\begin{table}[H]
	\centering
	\scalebox{0.6}{
	\begin{tabular}{ | r | r r r | r r r |  r r r |  r r r |  r r r |}
  \hline
  \multicolumn{1}{|r|}{ } & \multicolumn{3}{c|}{$\mathcal{RP}$}  & \multicolumn{3}{c|}{$\mathcal{RG}$} 
                          & \multicolumn{3}{c|}{$\mathcal{FG}$}
                          & \multicolumn{3}{c|}{$\mathcal{ARP}$} & \multicolumn{3}{c|}{$\mathcal{ES}$}	\\
  acc.  &  min & max & mean & min & max & mean & min & max & mean & min & max & mean & min & max & mean\\
  \hline 
$6.25 \cdot 10^{-2}$ &      0 &     0 &     0  &       0 &     0 &     0 &      0 &    0 &    0 &      0 &    0 &    0 &      0 &    0 &    0 \\
$3.12 \cdot 10^{-2}$ &      0 &     0 &     0  &       0 &     0 &     0 &      0 &    0 &    0 &      0 &    0 &    0 &      0 &    0 &    0 \\
$1.56 \cdot 10^{-2}$ &      0 &     0 &     0  &       0 &     0 &     0 &      0 &    0 &    0 &      0 &    0 &    0 &      0 &    0 &    0 \\
$7.81 \cdot 10^{-3}$ &      5 &     8 &     7  &       7 &    10 &     8 &      5 &    7 &    6 &      4 &   12 &    7 &      2 &    4 &    3 \\
$3.91 \cdot 10^{-3}$ &     22 &    33 &    27  &      36 &    45 &    42 &     17 &   25 &   21 &     29 &  163 &   66 &      6 &   10 &    9 \\
$1.95 \cdot 10^{-3}$ &     73 &   123 &    96  &     147 &   169 &   158 &     58 &   80 &   67 &    102 &  604 &  221 &     22 &   31 &   26 \\
$9.77 \cdot 10^{-4}$ &    282 &   410 &   344  &     511 &   566 &   538 &    110 &  167 &  140 &    198 &  690 &  375 &     73 &  102 &   86 \\
$4.88 \cdot 10^{-4}$ &    930 &  1304 &  1094  &    1579 &  1675 &  1621 &    208 &  305 &  261 &    301 & 1264 &  584 &    233 &  292 &  257 \\
$2.44 \cdot 10^{-4}$ &   2440 &  3232 &  2788  &    3987 &  4130 &  4045 &    327 &  517 &  401 &    328 & 1716 &  866 &    577 &  700 &  633 \\
$1.22 \cdot 10^{-4}$ &   4996 &  6186 &  5588  &    7909 &  8106 &  8009 &    448 &  656 &  557 &    574 & 2132 & 1332 &   1142 & 1344 & 1249 \\
$6.10 \cdot 10^{-5}$ &   8225 &  9610 &  8942  &   12697 & 12930 & 12824 &    586 &  765 &  696 &    763 & 2864 & 1621 &   1867 & 2101 & 1998 \\
$3.05 \cdot 10^{-5}$ &  11739 & 13117 & 12445  &   17699 & 17928 & 17833 &    739 &  855 &  803 &    903 & 4485 & 2372 &   2641 & 2891 & 2780 \\
$1.53 \cdot 10^{-5}$ &  15220 & 16645 & 15948  &   22750 & 22963 & 22870 &    795 & 1225 &  885 &   1004 & 5201 & 3010 &   3406 & 3692 & 3563 \\
$7.63 \cdot 10^{-6}$ &  18678 & 20032 & 19423  &   27772 & 28033 & 27917 &    900 & 1788 & 1355 &   1410 & 6242 & 3979 &   4223 & 4461 & 4348 \\
$3.81 \cdot 10^{-6}$ &  22154 & 23511 & 22902  &   32801 & 33091 & 32963 &    926 & 1870 & 1775 &   2145 & 7381 & 4691 &   5008 & 5254 & 5133 \\
$1.91 \cdot 10^{-6}$ &  25520 & 27034 & 26351  &   37844 & 38150 & 38008 &   1785 & 1939 & 1885 &   2199 & 8149 & 5609 &   5766 & 6050 & 5916 \\
  \hline  
\end{tabular}	

	}
  \caption{Nesterov smooth $f_3$, $L=1000$, $S=10833$, $n=64$. \#FES$/n$.}
	\label{tab:nesterovsmooth_f}
\end{table}

\begin{table}[H]
	\centering
	\scalebox{0.6}{
	\begin{tabular}{ | r | r r r | r r r |  r r r |  r r r |  r r r |}
  \hline
  \multicolumn{1}{|r|}{ } & \multicolumn{3}{c|}{$\mathcal{RP}$}  & \multicolumn{3}{c|}{$\mathcal{RG}$} 
                          & \multicolumn{3}{c|}{$\mathcal{FG}$}
                          & \multicolumn{3}{c|}{$\mathcal{ARP}$} & \multicolumn{3}{c|}{$\mathcal{ES}$}	\\
  acc.  &  min & max & mean & min & max & mean & min & max & mean & min & max & mean & min & max & mean\\
  \hline 
$6.25 \cdot 10^{-2}$ &      8 &    15 &    12 &      13 &    18 &    15 &    8 &    11 &   9 &      8 &   52 &   15 &      3 &    5 &    4  \\
$3.12 \cdot 10^{-2}$ &     29 &    45 &    35 &      50 &    61 &    56 &    21 &   32 &  26 &     43 &  173 &   93 &      8 &   13 &   11  \\
$1.56 \cdot 10^{-2}$ &     74 &   120 &    97 &     157 &   179 &   164 &    54 &   75 &  66 &     83 &  421 &  193 &     19 &   35 &   27  \\
$7.81 \cdot 10^{-3}$ &    198 &   342 &   255 &     397 &   438 &   408 &    93 &  139 & 119 &    184 &  497 &  290 &     52 &   76 &   65  \\
$3.91 \cdot 10^{-3}$ &    490 &   718 &   570 &     831 &   915 &   864 &   148 &  214 & 175 &    223 &  555 &  391 &    109 &  152 &  135  \\
$1.95 \cdot 10^{-3}$ &    969 &  1258 &  1071 &    1543 &  1648 &  1583 &   208 &  281 & 237 &    373 &  896 &  516 &    214 &  272 &  247  \\
$9.77 \cdot 10^{-4}$ &   1578 &  2012 &  1757 &    2495 &  2682 &  2568 &   276 &  354 & 314 &    477 &  989 &  662 &    352 &  447 &  399  \\
$4.88 \cdot 10^{-4}$ &   2374 &  2877 &  2605 &    3681 &  3929 &  3800 &   338 &  433 & 383 &    530 & 1144 &  763 &    533 &  638 &  587  \\
$2.44 \cdot 10^{-4}$ &   3260 &  3883 &  3607 &    5079 &  5387 &  5242 &   389 &  519 & 457 &    675 & 1209 &  876 &    740 &  875 &  810  \\
$1.22 \cdot 10^{-4}$ &   4315 &  5011 &  4710 &    6653 &  7026 &  6840 &   487 &  587 & 532 &    727 & 1370 & 1037 &    976 & 1132 & 1059  \\
$6.10 \cdot 10^{-5}$ &   5417 &  6220 &  5882 &    8347 &  8774 &  8555 &   563 &  671 & 612 &    978 & 1437 & 1179 &   1233 & 1403 & 1325  \\
$3.05 \cdot 10^{-5}$ &   6621 &  7435 &  7116 &   10113 & 10577 & 10337 &   642 &  752 & 685 &   1074 & 1782 & 1323 &   1508 & 1681 & 1603  \\
$1.53 \cdot 10^{-5}$ &   7872 &  8772 &  8371 &   11937 & 12411 & 12159 &   702 &  804 & 753 &   1212 & 1882 & 1468 &   1788 & 1974 & 1885  \\
$7.63 \cdot 10^{-6}$ &   9069 & 10005 &  9628 &   13786 & 14275 & 14002 &   766 &  861 & 818 &   1286 & 1986 & 1559 &   2071 & 2269 & 2173  \\
$3.81 \cdot 10^{-6}$ &  10331 & 11264 & 10885 &   15633 & 16129 & 15852 &   845 &  906 & 870 &   1448 & 2078 & 1687 &   2357 & 2568 & 2462  \\
$1.91 \cdot 10^{-6}$ &  11629 & 12482 & 12122 &   17455 & 17990 & 17708 &   883 & 1069 & 916 &   1557 & 2134 & 1825 &   2651 & 2854 & 2751  \\
  \hline  
\end{tabular}	

	}
  \caption{Nesterov strongly convex function $f_4$, $m=1$, $L=1000$, $S=1000$ $n=64$. \#FES$/n$.}
	\label{tab:nesterovstrong_f}
\end{table}

\begin{table}[H]
	\centering
	\scalebox{0.6}{
	\begin{tabular}{ | r | r r r | r r r |  r r r |  r r r |  r r r |}
  \hline
  \multicolumn{1}{|r|}{ } & \multicolumn{3}{c|}{$\mathcal{RP}$}  & \multicolumn{3}{c|}{$\mathcal{RG}$} 
                          & \multicolumn{3}{c|}{$\mathcal{FG}$}
                          & \multicolumn{3}{c|}{$\mathcal{ARP}$} & \multicolumn{3}{c|}{$\mathcal{ES}$}	\\
  acc.  &  min & max & mean & min & max & mean & min & max & mean & min & max & mean & min & max & mean\\
  \hline 
$6.25 \cdot 10^{-2}$ &   37 &  58 &  45 &   - & - & - &   - & - & - &    40 &  51 &  46 &   13 & 17 & 14 \\
$3.12 \cdot 10^{-2}$ &   63 &  86 &  71 &   - & - & - &   - & - & - &    62 &  76 &  70 &   19 & 25 & 22 \\ 
$1.56 \cdot 10^{-2}$ &   83 & 110 &  92 &   - & - & - &   - & - & - &    84 & 102 &  92 &   24 & 32 & 27 \\
$7.81 \cdot 10^{-3}$ &  103 & 129 & 112 &   - & - & - &   - & - & - &   103 & 123 & 112 &   29 & 37 & 32 \\
$3.91 \cdot 10^{-3}$ &  118 & 150 & 132 &   - & - & - &   - & - & - &   123 & 144 & 132 &   32 & 42 & 36 \\ 
$1.95 \cdot 10^{-3}$ &  140 & 168 & 151 &   - & - & - &   - & - & - &   140 & 167 & 151 &   35 & 46 & 40 \\
$9.77 \cdot 10^{-4}$ &  159 & 187 & 171 &   - & - & - &   - & - & - &   159 & 191 & 171 &   39 & 50 & 43 \\
$4.88 \cdot 10^{-4}$ &  178 & 207 & 190 &   - & - & - &   - & - & - &   177 & 212 & 192 &   43 & 54 & 47 \\
$2.44 \cdot 10^{-4}$ &  196 & 230 & 210 &   - & - & - &   - & - & - &   197 & 231 & 211 &   47 & 58 & 51 \\
$1.22 \cdot 10^{-4}$ &  215 & 252 & 232 &   - & - & - &   - & - & - &   218 & 265 & 234 &   51 & 62 & 55 \\
$6.10 \cdot 10^{-5}$ &  236 & 270 & 252 &   - & - & - &   - & - & - &   239 & 289 & 255 &   55 & 66 & 59 \\
$3.05 \cdot 10^{-5}$ &  257 & 293 & 273 &   - & - & - &   - & - & - &   257 & 313 & 275 &   59 & 70 & 63 \\
$1.53 \cdot 10^{-5}$ &  279 & 316 & 294 &   - & - & - &   - & - & - &   277 & 330 & 295 &   62 & 74 & 67 \\
$7.63 \cdot 10^{-6}$ &  297 & 340 & 316 &   - & - & - &   - & - & - &   295 & 355 & 317 &   67 & 77 & 70 \\
$3.81 \cdot 10^{-6}$ &  320 & 365 & 339 &   - & - & - &   - & - & - &   318 & 378 & 338 &   71 & 81 & 74 \\
$1.91 \cdot 10^{-6}$ &  338 & 384 & 360 &   - & - & - &   - & - & - &   342 & 399 & 361 &   73 & 85 & 78 \\
  \hline  
\end{tabular}	

	}
  \caption{Funnel function $f_5$, $S=32$ $n=64$. \#FES$/n$.}
	\label{tab:funnel_f}
\end{table}

\subsection{Different line search parameters for fixed dimension $n=64$}
Tables~\ref{tab:mu_i} and~\ref{tab:mu_f} summarize the number of function evaluations needed by $\mathcal{RP}_\mu$ on the ellipsoid function $f_2$ to achieve a relative accuracy of $1.91 \cdot 10^{-6}$ for different parameters $\mu$ that were passed to the used Matlab line search (cf. Figure~\ref{fig:rp}).

\begin{table}[H]
	\centering
	\scalebox{0.6}{
	\begin{tabular}{ | r | r | r | r | r | r | r | r | r | r | r |}
  \hline
  \multicolumn{1}{|r|}{ acc. } & \multicolumn{1}{c|}{1\textsc{e}$-1$} 
                          & \multicolumn{1}{c|}{1\textsc{e}$-2$} 
                          & \multicolumn{1}{c|}{1\textsc{e}$-3$} 
                          & \multicolumn{1}{c|}{1\textsc{e}$-4$} 
                          & \multicolumn{1}{c|}{1\textsc{e}$-5$} 
                          & \multicolumn{1}{c|}{1\textsc{e}$-6$} 
                          & \multicolumn{1}{c|}{1\textsc{e}$-7$} 
                          & \multicolumn{1}{c|}{1\textsc{e}$-8$} 
                          & \multicolumn{1}{c|}{1\textsc{e}$-9$} 						  
                          & \multicolumn{1}{c|}{1\textsc{e}$-10$} \\
  \hline 
$6.25 \cdot 10^{-2}$ &  2     &  2     & 2    &  2     & 2    & 2    & 2      &  2     & 2     & 2   \\
$3.12 \cdot 10^{-2}$ &  3     &  3     & 3    &  3     & 3    & 3    & 3      &  3     & 3     & 3   \\
$1.56 \cdot 10^{-2}$ &  3     &  3     & 3    &  3     & 3    & 3    & 3      &  3     & 3     & 3   \\
$7.81 \cdot 10^{-3}$ &  57    &  50    & 56   &  65    & 53   & 53   & 53     &  53    & 73    & 73   \\
$3.91 \cdot 10^{-3}$ &  201   &  194   & 219  &  225   & 210  & 210  & 210    &  210   & 232   & 232   \\
$1.95 \cdot 10^{-3}$ &  360   &  353   & 382  &  389   & 373  & 373  & 373    &  373   & 395   & 395   \\
$9.77 \cdot 10^{-4}$ &  523   &  516   & 546  &  552   & 536  & 536  & 536    &  536   & 558   & 558   \\
$4.88 \cdot 10^{-4}$ &  685   &  678   & 709  &  715   & 700  & 700  & 700    &  700   & 721   & 721   \\
$2.44 \cdot 10^{-4}$ &  849   &  842   & 871  &  878   & 862  & 862  & 862    &  862   & 883   & 883   \\
$1.22 \cdot 10^{-4}$ &  1011  &  1004  & 1035 &  1041  & 1024 & 1024 & 1024   &  1024  & 1045  & 1045   \\
$6.10 \cdot 10^{-5}$ &  1173  &  1166  & 1198 &  1203  & 1187 & 1187 & 1187   &  1187  & 1208  & 1208   \\
$3.05 \cdot 10^{-5}$ &  1335  &  1330  & 1360 &  1366  & 1350 & 1350 & 1350   &  1350  & 1370  & 1370   \\
$1.53 \cdot 10^{-5}$ &  1497  &  1494  & 1523 &  1528  & 1512 & 1512 & 1512   &  1512  & 1533  & 1533   \\
$7.63 \cdot 10^{-6}$ &  1661  &  1657  & 1686 &  1691  & 1675 & 1675 & 1675   &  1675  & 1696  & 1696   \\
$3.81 \cdot 10^{-6}$ &  1824  &  1819  & 1849 &  1853  & 1837 & 1837 & 1837   &  1837  & 1858  & 1858   \\
$1.91 \cdot 10^{-6}$ &  1986  &  1982  & 2013 &  2017  & 2001 & 2001 & 2001   &  2001  & 2020  & 2020   \\
  \hline  
\end{tabular}	

	}
  \caption{Different line search parameters $\mu$ for $\mathcal{RP}_\mu$ on ellipsoid function $f_2$, $m=1$, $L=1000$, $S=3200$, $n=64$, mean of 25 runs of \#ITS/$n$ to reach to reach a relative accuracy of $1.91 \cdot 10^{-6}$.}
	\label{tab:mu_i}
\end{table}

\begin{table}[H]
	\centering
	\scalebox{0.6}{
	\begin{tabular}{ | r | r | r | r | r | r | r | r | r | r | r |}
  \hline
  \multicolumn{1}{|r|}{ acc. } & \multicolumn{1}{c|}{1\textsc{e}$-1$} 
                          & \multicolumn{1}{c|}{1\textsc{e}$-2$} 
                          & \multicolumn{1}{c|}{1\textsc{e}$-3$} 
                          & \multicolumn{1}{c|}{1\textsc{e}$-4$} 
                          & \multicolumn{1}{c|}{1\textsc{e}$-5$} 
                          & \multicolumn{1}{c|}{1\textsc{e}$-6$} 
                          & \multicolumn{1}{c|}{1\textsc{e}$-7$} 
                          & \multicolumn{1}{c|}{1\textsc{e}$-8$} 
                          & \multicolumn{1}{c|}{1\textsc{e}$-9$} 						  
                          & \multicolumn{1}{c|}{1\textsc{e}$-10$} \\
  \hline 
$6.25 \cdot 10^{-2}$ &  15     &  16     &  16     &  17    & 16    &  16    & 16      &  16     &  16     & 16   \\
$3.12 \cdot 10^{-2}$ &  18     &  20     &  20     &  21    & 19    &  19    & 19      &  19     &  20     & 20   \\
$1.56 \cdot 10^{-2}$ &  24     &  25     &  25     &  26    & 25    &  25    & 25      &  25     &  26     & 26   \\
$7.81 \cdot 10^{-3}$ &  542    &  482    &  657    &  815   & 662   &  662   & 662     &  662    &  969    & 970   \\
$3.91 \cdot 10^{-3}$ &  1973   &  1908   &  2664   &  2983  & 2803  &  2805  & 2805    &  2806   &  3378   & 3390   \\
$1.95 \cdot 10^{-3}$ &  3565   &  3495   &  4651   &  5273  & 5125  &  5131  & 5131    &  5131   &  6140   & 6173   \\
$9.77 \cdot 10^{-4}$ &  5187   &  5120   &  6544   &  7558  & 7458  &  7469  & 7469    &  7470   &  9078   & 9126   \\
$4.88 \cdot 10^{-4}$ &  6813   &  6750   &  8319   &  9867  & 9848  &  9866  & 9867    &  9869   &  12224  & 12277   \\
$2.44 \cdot 10^{-4}$ &  8447   &  8380   &  10005  &  12191 & 12278 &  12307 & 12310   &  12311  &  15528  & 15582   \\
$1.22 \cdot 10^{-4}$ &  10069  &  10008  &  11653  &  14484 & 14705 &  14750 & 14754   &  14756  &  18872  & 18926   \\
$6.10 \cdot 10^{-5}$ &  11687  &  11627  &  13285  &  16693 & 17144 &  17212 & 17218   &  17220  &  22238  & 22293   \\
$3.05 \cdot 10^{-5}$ &  13309  &  13264  &  14908  &  18817 & 19566 &  19664 & 19674   &  19676  &  25576  & 25633   \\
$1.53 \cdot 10^{-5}$ &  14935  &  14902  &  16537  &  20820 & 21970 &  22114 & 22128   &  22131  &  28935  & 28996   \\
$7.63 \cdot 10^{-6}$ &  16572  &  16531  &  18167  &  22698 & 24375 &  24582 & 24602   &  24606  &  32308  & 32373   \\
$3.81 \cdot 10^{-6}$ &  18198  &  18155  &  19798  &  24449 & 26733 &  27030 & 27059   &  27064  &  35655  & 35726   \\
$1.91 \cdot 10^{-6}$ &  19824  &  19781  &  21435  &  26120 & 29071 &  29495 & 29537   &  29542  &  38993  & 39070   \\
  \hline  
\end{tabular}	

	}
  \caption{Different line search parameters $\mu$ for $\mathcal{RP}_\mu$ on ellipsoid function $f_2$, $m=1$, $L=1000$, $S=3200$, $n=64$, mean of 25 runs of  \#FES/$n$ to reach to reach a relative accuracy of $1.91 \cdot 10^{-6}$.}
	\label{tab:mu_f}
\end{table}

\end{document}